\documentclass[a4paper, 11pt,  reqno]{amsart} %reqno for right equation number 
\usepackage[left=3cm,right=3cm,top=3cm,bottom=3cm]{geometry}% very customizable margins. Under some (rare) circumstances, should be loaded after hyperref
\usepackage[babel]{microtype} % many good lay-out/justification effects, see:
\usepackage[dvipsnames]{xcolor} % Colors support
\usepackage[english]{babel}               % [french, frenchb, english, ]
\usepackage[T1]{fontenc}    % fontenc is oriented to output, that is, what fonts to use for printing characters. 
\usepackage{lmodern}
\usepackage[utf8]{inputenc} % inputenc allows the user to input accented characters directly from the keyboard; 
 \usepackage{amsmath}      % loads amstext, amsbsy, amsopn but not amssymb
\usepackage{amssymb}      % may be redundant with amsmath
\usepackage{amsthm}
\theoremstyle{plain}
\newtheorem{theorem}{Theorem}[section]

\newtheorem*{proposition*}{Proposition}
\newtheorem*{theorem*}{Theorem}
\newtheorem{lemma}[theorem]{Lemma}
\newtheorem*{lemma*}{Lemma}

\theoremstyle{definition}

\newtheorem{remark}[theorem]{Remark}

\newtheorem{conjecture}{Conjecture}
\usepackage{mathrsfs}      % for the command \mathscr{}
\usepackage{tcolorbox}
\usepackage{tikz} 
\usepackage{enumerate}    
\usepackage{enumitem}     % supports enumerate options such as label=(\roman*)
\usepackage{mathtools} 
\usepackage{textcomp}  % provides additional text symbols, including currency signs, trademark symbols, copyright symbols, and other special characters for use in documents.
\mathtoolsset{showonlyrefs} %hides equation numbers for all equations except those that are referenced in the document. 
\numberwithin{equation}{section}  % reset equation counters at start of each "section" and prefix numbers by section number
\usepackage[symbol]{footmisc}
 \usepackage{comment}      % provide new {comment} environment: all text inside the environment is ignored.
 \usepackage{appendix}
\usepackage{graphicx}  
\usepackage{float}                      % Improved interface for floating objects ; add [H] option

\usepackage[
breaklinks   = true,        % so long urls are correctly broken across lines
pagebackref  = true,     % add page number in bibliography and link to position in document where cited
]{hyperref}

\hypersetup{
	colorlinks=true,
	pdfpagemode=UseNone,
    citecolor=Green,
    linkcolor=OrangeRed,
    urlcolor=black,
	pdfstartview=FitW
}
\def\N{\mathbb{N}}

\def\R{\mathbb{R}}
\def\E{\mathbf{E}}
\def\P{\mathbf{P}}

\def\L{\mathcal{L}}
\newcommand{\ind}[1]{\mathbf{1}_{\left\{ #1 \right\}}}    %Indicator function 

\renewcommand{\bar}[1]{\mkern 1.5mu\overline{\mkern-1.5mu#1\mkern-1.5mu}\mkern 1.5mu}

\title[Left Tail of the Subcritical Derivative Martingale]{Left Tail of the Derivative Martingale in a Gaussian BRW
in the Entire Subcritical Regime}

\author{Xinxin Chen, Yichao Huang and Heng Ma}

\address[Xinxin Chen]
{Beijing Normal University, School of Mathematical Sciences, China}
\email{xinxin.chen@bnu.edu.cn}

\address[Yichao Huang]
{Beijing Institute of Technology, School of Mathematics and Statistics, China}
\email{yichao.huang@bit.edu.cn}

\address[Heng Ma]
 {School of Mathematics and Statistics, Beijing Institute of Technology, China \& Faculty of Data and Decision Sciences, Technion–Israel Institute of Technology, Israel.} 
\email{hengmamath@gmail.com}
\urladdr{\url{hengmamath.github.io}}

\begin{document}

\begin{abstract}
We establish a rather sharp two-sided estimate for the left tail probability of the derivative martingale limit in a binary Gaussian branching random walk throughout the entire subcritical regime, confirming a conjecture by Lacoin, Rhodes, and Vargas (\emph{Duke Math. J.} 171(3):483--545, 2022) in the case of Gaussian multiplicative cascades.
\end{abstract}

\maketitle

 %\tableofcontents 

\section{Introduction}
Consider a branching random walk (BRW) with binary splitting and independent Gaussian increments. It is a Gaussian  process $(S_{u}:u \in \mathbb{T})$ indexed by the binary tree $\mathbb{T}$ rooted at $\rho$. For each vertex $u \in \mathbb{T}$, we denote by $|u|$ the generation of $u$, defined as the graph distance between the root vertex $\rho$ and $u$. We write $v \prec u$ if $v$ is an ancestor of $u$, and $v \preceq u$ if $v \prec u$ or $v=u$. 

Say the root vertex $\rho$ is located at $S_\rho\in\R$. Let $(\xi_v)_{v \in \mathbb{T}}$ be a family of  i.i.d. copies of a standard Gaussian random variable $\xi$. For any $u \in \mathbb{T}$, set $S_u \coloneqq \sum_{\rho \prec v \preceq u} \xi_v+S_\rho$. Usually ,we take $S_\rho=0$.  Then, by interpreting $S_u$ as the energy of the vertex $u$, we define the normalized partition function of the corresponding Gibbs measure on the $n$-th generation at inverse temperature $\beta \geq 0$ as
\begin{equation*}
  W_{n}(\beta)\coloneqq\frac{\sum_{|u|=n}e^{\beta S_u}}{\E\left[\sum_{|u|=n} e^{\beta S_u}\right]}=\frac{1}{2^{n}}\sum_{|u|=n} e^{\beta S_u-\frac{\beta^2}{2}n},\quad n \ge 0. 
\end{equation*}
Thanks to the branching structure, the sequence $(W_{n}(\beta))_{n \ge 0}$  forms a non-negative martingale, known as the \emph{additive martingale}, and thus converges almost surely to some non-negative limit $W_{\infty}(\beta)$. 
 It is well known from Biggins' martingale convergence theorem
 (see e.g., \cite{Biggins77} and \cite[Theorem 3.2]{Shi15}) that there exists a critical inverse temperature $\beta_{c} := \sqrt{2 \ln 2}$ such that the limit $W_{\infty}(\beta)$ is non-degenerate if and only if $\beta \in [0, \beta_c)$, and in this case, $W_\infty(\beta)>0$ a.s.   
By differentiating $W_{n}(\beta)$ with respect to $\beta$, we obtain the so-called \emph{derivative martingale} of the branching random walk:
\begin{equation}
  D_n(\beta) \coloneqq   \frac{\partial}{\partial \beta} W_{n}(\beta) = \frac{1}{2^{n}}\sum_{|u|=n} (S_u- \beta n  ) e^{\beta S_u - \frac{\beta^2}{2} n}  ,\quad n \ge 0. 
\end{equation} 
%The usual notation of the derivative martingale $D_n(\beta)=-Z_n(\beta)$ comes with an opposite sign. 
It can be verified that  for $\beta \in [0,\beta_{c})$,  $D_{n}(\beta)$ is a signed martingale and converges a.s. to a non-degenerate limit $D_{\infty}(\beta)$, see e.g. \cite[Proposition 3.2]{CdRM24}.

The right tail behavior of the limit $D_{\infty}(\beta)$ was previously studied in \cite{CdRM24}. Specifically, for fixed $\beta \in (0,\beta_{c})$ there exists some constant $C_{\beta}>0$ such that 
\begin{equation}\label{eq-tail-of-D}
  \P(  D_{\infty}(\beta)  > y) \sim C_{\beta} \Bigl( \frac{\ln y}{y}     \Bigr)^{\gamma } \text{ as } y \to \infty 
\end{equation} 
where the exponent $\gamma=\left({\beta_{c}}/{\beta}\right)^{2}$ is always greater than $1$ in the subcritical regime.

Our first theorem establishes left tail probability estimates for the limiting derivative martingale $D_{\infty}(\beta)$. This confirms~\cite[Conjecture 1]{LRV22}\footnote{For readers without access to the published version, the detailed statement is available in the HAL version.} for Gaussian multiplicative cascades (see also~\cite{BV23}). 
%With our notation, our theorem can be stated as a right tail estimate on the limit $Z_\infty(\beta)$:
\begin{theorem}\label{thm-Z-tail}
For any $\beta \in (0, \beta_{c})$, write $\gamma=\left({\beta_{c}}/{\beta}\right)^{2}>1$ as above. There exist constants $0<c<C<\infty$ depending only on $\beta$  such that for any $y \ge 1$, 
\begin{equation}\label{eq-Z-tail}
c\exp\left(-Cy^{\gamma}\right)\leq\P(D_{\infty}(\beta)<-y)\leq C\exp\left(-cy^{\gamma}\right). 
\end{equation}
\end{theorem}
 
Our second result establishes that the ratio of the derivative martingale limit with respect to the additive  martingale limit ${D_{\infty}(\beta)}/{W_{\infty}(\beta)}$ has finite negative exponential moments of all orders. This extends \cite[Theorem 3.2(ii)]{LRV22} to the entire subcritical regime $\beta<\beta_{c}$ for log-correlated Gaussian fields on trees,   whereas in \cite{LRV22} the corresponding statement was established for the derivative Gaussian multiplicative chaos (DGMC) only for $\beta\in(0,\frac{\beta_c}{2})$. 

\begin{theorem}\label{thm-Z/W-tail}
For any $\beta\in(0, \beta_{c})$, there exist $0<c<\infty$ depending only on $\beta$ such that for any $y \ge 1$,
\[  \P  \Bigl(  \frac{ D_{\infty}(\beta) }{W_{\infty}(\beta)} <- y   \Bigr) \leq   \exp \bigl( - c y^{\gamma \wedge 2}    \bigr).   \]  
 In particular,  the ratio $\frac{D_{\infty}(\beta)}{W_{\infty}(\beta)}$ admits all negative exponential moments:
\[
  \E\left[\exp\left\{-K\frac{D_{\infty}(\beta)}{W_{\infty}(\beta)}\right\}\right]<\infty,
  \qquad \forall K>0.
\]
\end{theorem}

\begin{remark}\label{rmk:enhanced-thm}
Our method yields a stronger, uniform version of the bound in Theorem~\ref{thm-Z-tail}. In particular, our result shows that there exist constants $C,c>0$ depending
only on $\beta$ such that 
\begin{equation}
  \label{eq:enhanced-thm}
  \sup_{x\in\mathbb{R}}\P(D_{\infty}^{[x]}(\beta)<-y)\leq C\exp\left(-cy^{\gamma}\right),\quad y\geq1.
\end{equation}
Here $D_{\infty}^{[x]}(\beta)$ denotes the derivative martingale limit of the BRW starting at $S_\rho=x \in \mathbb{R}$, see~\eqref{eq-Z-x-n-1} below. 
Furthermore, the identity $\frac{D_{\infty}^{[x]}(\beta)}{W_{\infty}^{[x]}(\beta)}=\frac{D_{\infty}(\beta)}{W_{\infty}(\beta)}+x$ implies that Theorem~\ref{thm-Z/W-tail} also holds for the branching random walk with arbitrary initial position, with the constant $c$ depending on the initial position.
\end{remark}

\subsection{Motivations and related works} 
The Gaussian branching random walk studied in this paper is, after exponentiation and martingale normalization, exactly Mandelbrot's lognormal multiplicative cascade, introduced in the study of energy dissipation in turbulent flows~\cite{Man72,Man74}. Kahane's theory of Gaussian multiplicative chaos (GMC)~\cite{Kahane85} later generalized this picture to the situation of a log-correlated Gaussian field on a continuous space. Both cascades and GMC share many basic features and exhibit the same phase transition in the inverse temperature parameter. GMC has recently gained significant attention for its role in fields such as Liouville conformal field theory, random matrix theory, and number theory. See~\cite{BP25,RV14} for more details. 
 
Based on the theory of GMC,  Lacoin, Rhodes, and Vargas~\cite{LRV22} rigorously defined the quantum Mabuchi-Liouville theory~\cite{BFK14,FKZ11,FKZ12}, by constructing non-perturbatively a finite measure on the space of Riemann metrics of a given compact Riemann surface. In physics, this corresponds to a path integral based on a Gibbs type measure with energy given by the sum of the Liouville action $\mathcal{S}_{\mathrm{L}}$ and the quantum Mabuchi K-energy $\mathcal{S}_{\mathrm{M}}$ with action $\bar{\beta} \mathcal{S}_{\mathrm{M}} + \mathcal{S}_{\mathrm{L}}$ (see~\cite[(1.5) and (1.10)]{LRV22}). The Liouville action leads to the GMC denoted as $\mathcal{G}_{\infty}(\bar{\gamma})$, while the potential  term  $( \bar{\gamma}\varphi )e^{ \bar{\gamma} \varphi}$ in the Mabuchi action corresponds to the derivative  of GMC, $\mathcal{D}_{\infty}(\bar{\gamma} ) = \frac{\partial}{\partial \bar{\gamma} } \mathcal{G}_{\infty}(\bar{\gamma})$.\footnote{Here $\bar{\beta}$ and $\bar{\gamma}$ are different from $\beta$ and $\gamma$ above, with $\bar{\beta}/\bar{\gamma}$ equal to some explicit positive constant.}

A crucial ingredient in the construction of the quantum Mabuchi theory is to show that these derivative GMC   random variables $\mathcal{D}_{\infty}(\bar{\gamma})$  possess all negative exponential moments, see~\cite[Theorem 3.4 and equation (3.9)]{LRV22}.  
The authors of~\cite{LRV22} established the sub-Gaussian tail probability bound
\begin{equation}\label{eq-LVR-est}
  \P( \mathcal{D}_{\infty}(\bar{\gamma}) \le -v  )\le 2  e^{- c v^2}
\end{equation} 
valid in the so-called $L^4$ regime, where the coupling parameter satisfies $\bar{\gamma} \in (0, \bar{\gamma}_{c}/2)$ with $\bar{\gamma}_{c}$ denoting the critical value. As emphasized in~\cite{LRV22}, the technical restrictions underlying estimate  \eqref{eq-LVR-est}, as well as the universality of derivative  GMC measures, prevent their statements from covering the whole   range of expected valid parameters $\bar{\gamma} \in (0, \bar{\gamma}_{c})$. More precisely, the core argument of~\cite{LRV22} is a concentration bound for exponential martingale, which requires the control of the quadratic variation of the derivative martingale. Since the latter is not always finite, a finely chosen cutoff is applied, and the quadratic variation method is again used to control the error term in the cutoff. The double quadratic variations lead ultimately to the $L^{4}$ constraint of this approach.

Motivated by this limitation, the authors conjectured that the sub-Gaussian bound can be sharpened to hold throughout the entire subcritical interval $\bar{\gamma} \in (0, \bar{\gamma}_{c})$:  
\begin{conjecture}[{\cite{LRV22}}] 
For any log-correlated Gaussian field defined on $D \subset \mathbb{R}^d$ admitting a smooth white noise decomposition, and for every $\bar{\gamma} \in (0, \bar{\gamma}_{c} )$,
\begin{equation}
-\ln \P(\mathcal{D}_\infty(\bar{\gamma}) < -v) \asymp v^{(\bar{\gamma}_{c}/\bar{\gamma})^{2}}. 
\end{equation} 
\end{conjecture} 
To tackle this conjecture, one can consider the same question for Gaussian multiplicative cascades. In this context, Bonnefont and Vargas established in~\cite{BV23} the corresponding upper bounds for the left tail probability of $D_{\infty}(\beta)$ in the $L^4$ regime  $\beta \in (0,\beta_{c}/2)$, based on the concentration method of~\cite{LRV22} as explained above. They also proved the conjectured lower bound in the whole $L^{1}$ regime $\beta \in (0,\beta_{c})$.

Our main result, Theorem~\ref{thm-Z-tail}, fills the gap for the upper bound when $\beta$ is in the regime $ [\beta_{c}/2, \beta_{c})$. This requires completely abandoning the quadratic variation and concentration method approach of~\cite{LRV22} and~\cite{BV23}. We rely on a general recursive scheme to control the moment growth of the negative part $(D_{\infty}(\beta))_-$, only with very soft (polynomial tail) a priori estimates given by the study of level sets of BRW to initiate the iteration. Our result and method, valid without restriction on $\beta\in(0,\beta_c)$, provides strong evidence towards the validity of the above conjecture for GMC and the probabilistic construction in~\cite{LRV22} in the entire subcritical regime, although this more general case requires some further technical inputs and is still under investigation.

% While this paper considers only the derivative martingales in the subcritical regime $\beta\in(0,\beta_c)$, the endpoint cases are also well understood. When $\beta=0$, the limit $Z_\infty(0)$ is Gaussian and the tail estimate is elementary, see  \cite[Remark 1.4]{BV23}. At criticality $\beta=\beta_c$, the derivative martingale limit exists almost surely, is nonnegative, and has Cauchy tail; see, for instance, Buraczewski~\cite{Bur09}. The behavior near zero was studied by Hu~\cite{Hu16}. For further properties about derivative martingales at criticality, we refer to \cite{BKNSW15,DRSV14,Lac24,Pow21,Remy20} for GMC and \cite{Aid13,AS14,BIM21,Madaule16,Mad16} for BRW.

\subsection{Proof ideas} 
For convenience, we consider $Z_n(\beta)=-D_n(\beta)$ and $Z_\infty(\beta)=-D_\infty(\beta)$ in the following, and study the decaying rate of the probability $\P(Z_\infty(\beta)>y)$.
%We use the standard asymptotic notation $\Theta(\cdot)$ in this subsection, see Section~\ref{sec-notation}. 
We begin by recalling some heuristics behind Theorem~\ref{thm-Z-tail} given in~\cite{LRV22,BV23}. The strategy to make $Z_{\infty}(\beta)$ large is to force
\begin{equation}
  Z_{n}(\beta)=\frac{1}{2^{n}}\sum_{|u|=n}(\beta n-S_u)e^{\beta S_u-\frac{\beta^2}{2}n}
\end{equation}
to be large at a carefully chosen   generation $n$. Observe that
$ Z_{n}(\beta)$ is bounded from above  by $ e^{{n\beta^2 }/{2}}\sup_{x \in \mathbb{R}}(xe^{-\beta x})=(\beta e)^{-1}e^{n\beta^2/2},$
 with this maximum occurring when all the particles at generation $n$ are positioned at level $\beta n-{1}/{\beta}$. Thus, to ensure that the event $\{Z_{\infty}(\beta)>y\}$ occurs,  we select $n$ such that   $e^{\beta^2 n /2}\approx y$, then place all the particles at generation $n$ within $\Theta(1)$ distance to $\beta n-1/\beta $\footnote{We say that $a_n=\Theta(b_n)$ if $c b_n \le a_n \le C b_n$ for some constants $0<c<C<\infty$.}. Conditioned on this event  and  on the first $n$ generations, both $Z^{(u)}_{\infty}(\beta)$ and  $W^{(u)}_{\infty}(\beta)-1$ have  mean zero and finite $p$-th moment for $p\in(1,\gamma)$.
A law-of-large-numbers heuristic then suggests that 
\begin{align*}
  &Z_{\infty}(\beta)-Z_{n}(\beta)=\frac{1}{2^{n}}\sum_{|u|=n}e^{\beta S_{u}-\frac{\beta^2}{2}n}\left[Z^{(u)}_{\infty}(\beta)+(\beta n-S_{u})(W^{(u)}_{\infty}(\beta)-1)\right]\\
  &\approx e^{\beta^2 n/2}\bigg[\frac{1}{2^{n}}\sum_{|u|=n}Z^{(u)}_{\infty}(\beta)+\Theta(1)\frac{1}{2^{n}}\sum_{|u|=n}(W^{(u)}_{\infty}(\beta)-1)\bigg]=o(e^{\beta^2 n/2}) . 
\end{align*}
This indicates that the strategy of placing all individuals at generation $n$ around $\beta n - 1/\beta$ yields an effective lower bound. This is achieved on the event $\{\forall v\in\mathbb{T}_{n-k},\quad \xi_{v}\in[\beta-\frac{1}{\beta }2^{-k}, \beta-\frac{1}{\beta}2^{-k-1}]\}$, whose probability cost is of order $\exp(- \Theta(2^{n}))$. By relating this cost back to $y$ through the identity $2^{n}=(e^{\beta^2 n/2})^{\gamma}\approx y^{\gamma}$, we obtain the desired lower bound for the tail probability.

The deriving of the matching upper bound for $\P( Z_{\infty} (\beta)> y)$ is more involved. Even though we suspect that the previous strategy is (nearly) optimal, we cannot rule out better alternatives. Instead, we estimate the growth rate of positive moments of  $(Z_{\infty}(\beta))_+:=\max\{ Z_{\infty}(\beta),0\}$. A starting point is that the branching property implies the following functional recursive relation for the (shifted) derivative martingale limits (see~\eqref{eq-Z-x-n-1}), 
\begin{equation}
  Z^{[x]}_{\infty} (\beta) = \frac{e^{\beta^2/2}}{2}  \sum_{|u|=1} Z^{[x+S_u-\beta],(u)}_{\infty}(\beta).
\end{equation} 
In particular, taking the positive part, we obtain that 
 \begin{equation}\label{recursive-ineq}
 ( Z^{[x]}_{\infty} (\beta) )_+  \le  \frac{e^{\beta^2/2}}{2}  \sum_{|u|=1} (Z^{[x+S_u-\beta],(u)}_{\infty}(\beta))_+ .
 \end{equation}
Employing an inductive argument, we prove in \S~\ref{sec-pf-thm1} that if for any integer $k\ge 1$, 
\begin{equation}\label{eq-finite-moment}
  \sup_{x\in\mathbb{R}}\E\left[(Z^{[x]}_{\infty}(\beta))_+^{k}\right]<\infty,
\end{equation}
then the recursive relation  \eqref{recursive-ineq} implies that there exists a constant $C>0$ such that 
\begin{equation}
  \sup_{x\in\mathbb{R}}\E[(Z_{\infty}^{[x]}(\beta))_+^{k}]\leq C^k k^{k/\gamma}\quad\text{as}\quad k\to\infty.
\end{equation}   
By applying the classical criterion for sub-exponential distributions (see Lemma~\ref{lem-sub-exp} below), we can conclude that $ (Z_{\infty}^{[x]}(\beta))_+^{\gamma}$ has an exponential tail,  which   leads to the desired result.  

\begin{remark}
We stress that the preliminary polynomial bound~\eqref{eq-finite-moment} is much softer than the super-exponential decay of our main result Theorem~\ref{thm-Z-tail}, and we expect this strategy to be applied to a number of similar models up to technical modifications.
\end{remark}

We establish the boundedness of an arbitrary positive moment of $(Z^{[x]}_{\infty}(\beta))_+$ uniformly in the initial position by showing that 
$\sup_{x\in\mathbb{R}}\P(Z^{[x]}_{\infty}(\beta)>y)$ decays faster than any polynomial rate in $y$ as $y\to\infty$. To this end, we again approximate $Z^{[x]}_{\infty} (\beta)$ by $Z^{[x]}_{n} (\beta)$   for a carefully chosen  generation $n$ (either $n=\Theta(\ln y)$ or $n=\Theta(x)$ depending on the relation between $x$ and $y$), and reduce the problem to
  proving the super-exponential decay for two types of probability estimates,  
\begin{equation}
  \P( Z_{n}(\beta)>e^{ n \Delta})\quad\text{and}\quad\P(Z_{n}(\beta)>\theta nW_{n}(\beta)) 
\end{equation}  for any small $\Delta>0, \theta>0$.
Our main lemmas~\ref{lem-Polynomial-decay-low-initial-position} and~\ref{lem-Polynomial-decay-high-initial-position} are devoted to showing that both probabilities  exhibit a decay rate  at least as fast as $e^{- \Theta( n^2)}$.  

We now sketch the proof idea for the first estimate in the above display, since the two estimates are handled similarly. We proceed by analyzing the  contribution in $Z_{n}(\beta)$ from the $I$-level sets, which counts the  individuals at generation $n$ located in a certain interval $I$. 
Notice that the $(\beta n,\infty)$-level set contributes negatively  in $Z_{n}(\beta)$ and the contribution of the $(-\infty, \beta n/2)$-level set  in  $Z_{n}(\beta) $ is negligible. The crucial observation is that, in order to have $ Z_{n}(\beta)>e^{\Delta n}$, there must exist some $1\le k\le\beta n/2$ such that
\begin{enumerate}[label=(\roman*)]
  \item the size of the $(\beta n-k, \beta n-k+1)$-level set  is unusually large: the exponential growth rate (Malthusian  exponent)  exceeds its expectation by some positive constant depending on $\Delta$;
  \item the  contribution from the $(\beta n-k, \beta n-k+1)$-level set beats the $\frac{2}{\beta n}$ fraction of the negative contribution from the $(\beta n,\infty)$-level set, imposing a constraint on the  growth rate of the latter.
\end{enumerate}
If we only consider the large deviation event described in (i), the probability cost is $e^{-\Theta(n)}$. However, as established in \cite{AHS19}, conditioned on this large deviation event,  with probability $1-e^{-\Theta(n^2)}$, the branching random walk must hit  a specific "target region" (shaded in Figure \ref{f1}). 
On the other hand, based on the study of the lower deviation probability in branching random walk models~\cite{CH20,Zhang23}, the second event incurs a probability cost of  $e^{-\Theta(n^2)}$ if the branching random walk hits a certain "forbidden region" (pink in Figure \ref{f1}).  Elementary calculations demonstrate that the target region is entirely contained within the forbidden region. This contradiction provides us with the desired decay rate.

\begin{figure}[tbp]
  \includegraphics[width=0.65\linewidth]{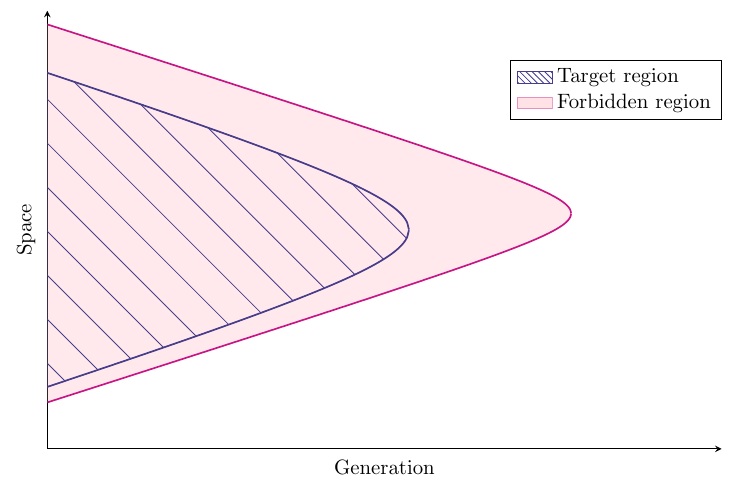}
  \caption{Target region and forbidden region.}\label{f1}
\end{figure}

\noindent\textbf{Outline of the paper.}
In~\S\ref{sec-lower-bnd}, we prove the lower bound for the tail probability $\P(Z_{\infty}(\beta)>y)$ as $y \to \infty$. In~\S\ref{sec:upper_bound_for_the_tail_probability}, we derive the corresponding upper bound of Theorem~\ref{thm-Z-tail}.  In~\S\ref{sec-LDP}, we present two large-deviation lemmas for level sets of a BRW. In~\S\ref{sec-poly-decay}, we show that $\P(Z_{\infty}(\beta)>y)$ decays faster than any polynomial rate in $y$. Finally, the proofs of Theorems~\ref{thm-Z-tail} and~\ref{thm-Z/W-tail} are gathered in~\S\ref{sec-pf-thm1} and~\S\ref{sec-pf-thm2}, respectively.

\subsection{Notation convention}\label{sec-notation} 
We fix some notation that will be used in the sequel. 
\begin{itemize}
  \item For each $n \ge 0$,  let $\mathcal{F}_{n}$ denote the $\sigma$-field generated by $(S_{u} : |u| \le n )$.  
  \item For any $v \in \mathbb{T}$, let $\mathbb{T}^{(v)}$ denote the subtree rooted at $v$. In particular $\mathbb{T}^{(\rho)}=\mathbb{T}$. For $n\geq 0$, $\mathbb{T}^{(v)}_{n}$ represents the $n$-th generation of $\mathbb{T}^{(v)}$ and $\mathbb{T}_{\le n}^{(v)}\coloneqq\cup_{k =0}^{n}\mathbb{T}^{(v)}_{k}$. Note that $\mathbb{T}^{(v)}_{k}\subset\mathbb{T}_{|v|+k}$. 
\item For any $v\in\mathbb{T}$, define the shifted process $S^{(v)}_{u}\coloneqq S_u-S_v$ for all $u\in\mathbb{T}^{(v)}$. By construction, conditionally on $ \mathcal{F}_{n}$, the random processes $(S^{(v)}_{u}:u\in\mathbb{T}^{(v)})_{|v|=n}$ are i.i.d. copies of the original branching random walk starting at zero. We refer to this property as the branching property. 
  \item For any functional $F$ of the branching random walk, we denote by $F^{(v)}$ the value of this functional acting on the process  $(S^{(v)}_{u}:u\in\mathbb{T}^{(v)})$. For example, for $n\in\mathbb{N}\cup\{\infty\}$, we define
\begin{equation}
  Z^{(v)}_{n}(\beta)\coloneqq\sum_{u\in\mathbb{T}^{(v)}_{n}}(\beta n-S^{(v)}_u)e^{\beta S^{(v)}_u-\Phi(\beta)n} \quad \text{ and }\quad Z^{(v)}_{\infty}(\beta) \coloneqq \lim_{n \to \infty} Z^{(v)}_{n}(\beta) 
\end{equation}
and similarly for $W^{(v)}_{n}(\beta)$,  $W^{(v)}_{\infty}(\beta)$.
\end{itemize}  
For $x \in \mathbb{R}$, define  $(x)_+ :=\max(x,0)$, write $\lceil x\rceil$ for the smallest integer surpassing $x$, and write $x\wedge y:=\min\{x,y\}$. For two nonnegative sequences $f(n), g(n)$, we write $f(n) \lesssim g(n)$ or $f(n)=O(g(n))$ if  there exists a positive constant $C$ such that $f(n) \le C g(n)$ for all $n \ge 1$. We write $f(n)= \Theta(g(n))$ if $f(n)=O(g(n))$ and $g(n)= O(f(n))$. The notation $f(n) =o(g (n))$ means that $f(n)/g(n) \to 0$ as $n \to \infty$.  We use $c$ and $C$ to represent positive constants which may vary from line to line.

%%%%%%%%%%%%%%%%%%%%%%%%%%%%%%%%%%%%%%%%%%%%%%%%%%%
%%%%%%%%%%%%%%%%%%%%%%%%%%%%%%%%%%%%%%%%%%%%%%%%%%%
%%%%%%%%%%%%%% Lower bound  %%%%%%%%%%%%%%%%%%%%%%%
%%%%%%%%%%%%%%%%%%%%%%%%%%%%%%%%%%%%%%%%%%%%%%%%%%%
%%%%%%%%%%%%%%%%%%%%%%%%%%%%%%%%%%%%%%%%%%%%%%%%%%%

\section{Lower bound for the tail probability}\label{sec-lower-bnd}
We now derive the lower bound for the tail probability in Theorem~\ref{thm-Z-tail}. Although this was previously established in~\cite{BV23}, we present a simplified proof which does not use the Gaussian step distribution. 
 
\begin{proof}[Proof of Theorem~\ref{thm-Z-tail}: Lower bound] Note that $Z_{n+k}(\beta)$ can be decomposed at the $n$-th generation of the branching random walk via the branching property:
\begin{equation*}
  Z_{n+k}(\beta)=\frac{1}{2^{n}}\sum_{|u|=n}e^{\beta S_{u}-\frac{\beta^2}{2}n}[Z^{(u)}_{k}(\beta)+(\beta n-S_{u})W^{(u)}_{k}(\beta)].
\end{equation*}
Taking the limit as $k\to\infty$, we obtain 
\begin{equation}
  Z_{\infty}(\beta)=\frac{1}{2^{n}}\sum_{|u|=n}e^{\beta S_{u}-\frac{\beta^2}{2}n}\left[ Z^{(u)}_{\infty}(\beta)+(\beta n-S_{u}) W^{(u)}_{\infty}(\beta)\right].
\end{equation} 
 
For any $n\geq0$, we introduce two ``good'' events $G_{n}^{1}\coloneqq \{  1/\beta \leq \beta n - S_{u}  \leq 2/\beta  \text{ for all }  |u|= n \}$ and $G_{n}^{2}\coloneqq \{  W^{(u)}_{\infty}(\beta)  \geq 1 , Z_{\infty}^{(u)}(\beta) \geq 0    \text{ for all }  |u|= n \}$. Note that on the event $G_{n}^{1} \cap G_{n}^{2}$, we have
\begin{equation}
  Z_{\infty}(\beta)\geq e^{\frac{\beta^2}{2}n}\cdot\beta^{-1}\cdot\min_{t\in [1,2]}(te^{-t})=e^{\frac{\beta^2}{2}n}\beta^{-1}\cdot 2e^{-2}. 
\end{equation}
Given a fixed $y >0$, we choose the smallest $n=n(y)\in\N$ satisfying $e^{\frac{\beta^2}{2}n}\geq\beta e^{2} y/2$. It follows that
\begin{equation}
   \P(Z_{\infty}(\beta) \geq y)\geq \P(G_{n}^{1}\cap G_{n}^{2}).
\end{equation} 
Recall that conditionally on  $\mathcal{F}_{n} \coloneqq \sigma(S_{u}:|u|\le n)$, the processes $\{(S^{(u)}):|u|=n\}$ are independent branching random walks with the same law as the original process. Let $p_{*}(\beta)\coloneqq\P(W_{\infty}(\beta)\geq 1, Z_{\infty}(\beta)\geq 0)>0$. It follows that  $\P(G^{2}_{n}\mid\mathcal{F}_{n})=p_{*}(\beta)^{2^{n}}$, which gives
\begin{equation}
  \P(Z_{\infty}(\beta)\geq y)\geq\P( G_{n}^{1} \cap G_{n}^{2})=\E\left[\P(G_{n}^{2}\mid \mathcal{F}_{n})\ind{G_{n}^{1}}\right] =  \P(G_{n}^{1})p_{*}(\beta)^{2^{n}}. 
\end{equation}
 
Next we bound the probability $\P(G_{n}^{1})$ from below. Fix $\eta=3/4$. For $0\leq k\leq n-1$, define 
\begin{equation}
   G^{1}_{n,k}\coloneqq  \Bigl\{ \,  \forall v\in\mathbb{T}_{n-k},\quad \xi_{v}\in \Bigl[\beta-\frac{1}{\beta2^{k}}, \beta-\frac{\eta}{\beta 2^{k}}   \Bigr]   \Bigr\} ,
\end{equation}
where $(\xi_v)_{v \in \mathbb{T}}$ are the i.i.d. Gaussians used to construct the branching random walks. On the event $\cap_{k=0}^{n-1}G^1_{n,k}$, every particle $u$ with $|u|=n$ satisfies
$\frac{2\eta}{\beta}(1-2^{-n})\leq \beta n-S_u\leq \frac{2}{\beta}(1-2^{-n}).$
Since $\eta>1/2$, for sufficiently large $n$ this gives $\bigcap_{k=0}^{n-1}G^{1}_{n,k}\subset G^{1}_{n}$. By independence of  $(\xi_v:v\in\mathbb{T})$, we have 
\begin{equation}
  \P\Bigl( \bigcap_{k=0}^{n-1}G^{1}_{n,k} \Bigr) =\prod_{k=0}^{n-1}\P(G^{1}_{n,k})=\prod_{k=0}^{n-1}\P\Bigl( \beta-\xi\in \Bigl[ \frac{\eta}{\beta2^{k}},\frac{1}{\beta 2^{k}}   \Bigr] \Bigr) ^{2^{n-k}}. 
\end{equation}
Since $\xi\sim\mathcal{N}(0,1)$ has a smooth density, there exists a constant $b>0$ depending on $\beta$ such that for all $k\geq 0$,
$
   \P ( \beta-\xi\in [ \frac{\eta}{\beta2^{k}},\frac{1}{\beta 2^{k}}   ] )  \geq\frac{1}{b 2^{k}}$.
Combining the above, we obtain 
\begin{equation}
  \P ( G^{1}_{n} ) \geq   \prod_{k=0}^{n-1}  \exp  \bigl(  - {2^{n-k}} \ln  ({b 2^{k}})  \bigr)  \geq   \exp  \Bigl(- 2^{n} \sum_{k=0}^{\infty} (k\ln 2+\ln b) 2^{-k}    \Bigr)   .
\end{equation}
As the series  $\sum_{k=0}^{\infty} (k\ln 2+\ln b) 2^{-k}  $ converges, we conclude that there exists some constant $c'>0$ depending only on $\beta$ such that $  \P(   Z_{\infty}(\beta) \geq y ) \ge  \P( G_{n}^{1} \cap G_{n}^{2}) 
  \ge    \exp \left(  - c' 2^{n}   \right)$. By our choice of $n=n(y)$,
\begin{equation}
   2^{n} = e^{ \frac{\beta^2}{2} n \cdot \frac{2 \ln 2}{\beta^2} } =  e^{ \frac{\beta^2}{2} n \cdot \gamma }  \leq   \Bigl(e^{\frac{\beta^2}{2}}\beta e^{2} /2\Bigr) ^{\gamma}  \,  y^{\gamma } .
\end{equation}  
This yields the desired result.
\end{proof} 
%%%%%%%%%%%%%%%%%%%%%%%%%%%%%%%%%%%%%%%%%%%%%%%%%%%
%%%%%%%%%%%%%%%%%%%%%%%%%%%%%%%%%%%%%%%%%%%%%%%%%%%
%%%%%%%%%%%%%% Upper Bound  %%%%%%%%%%%%%%%%%%%%%%%
%%%%%%%%%%%%%%%%%%%%%%%%%%%%%%%%%%%%%%%%%%%%%%%%%%%
%%%%%%%%%%%%%%%%%%%%%%%%%%%%%%%%%%%%%%%%%%%%%%%%%%%

\section{Upper bound for the  tail probability}\label{sec:upper_bound_for_the_tail_probability}

\subsection{Preliminaries: large deviation probabilities of level sets}\label{sec-LDP}
We first introduce two key lemmas that provide super-exponential decay estimates for certain rare events associated with the level sets of the branching random walk. For any interval $I\subset\mathbb{R}$, define 
\begin{equation}\label{eq:DefCountingL}
  \L_{n}^{[x]}(I)\coloneqq\sum_{|u|=n}\ind{x+S_u\in I},  
\end{equation}
which counts the number of particles at generation $n$ whose positions, shifted by $x$, fall within the interval $I$. If $I=(a,b)$ we write $  \L_{n}^{[x]}(a,b)$ for  $\L_{n}^{[x]}(I)$ for simplicity.  
Introduce the auxiliary function
\begin{equation}
  \varPhi(n,x)\coloneqq n\ln 2-\frac{x^2}{2n}\quad\text{for}~n\ge 1~\text{and}~x\in\mathbb{R}.
\end{equation}
Note that $\varPhi(n,\beta n)=\Phi(\beta)n$ where $\Phi(\beta)\coloneqq\ln 2-\frac{\beta^2}{2}$ for $\beta\in\mathbb{R}$.

The following lemma recasts a result from~\cite{AHS19} in the context of branching Brownian motion. A notable difference in our formulation is that the bound holds uniformly over both the level $L$ and the size of the level set (denoted $M$ below). Moreover, the convergence can be made super-exponentially fast in $n$, which is not highlighted in~\cite{AHS19}. For completeness, we include a proof in~\S\ref{sec-LDP-proof-b}.

\begin{lemma}[{\cite[Section~3.2]{AHS19}}]\label{lem-LDP-level-set-below-the-line}
Suppose  $n\geq 1$, $L\in\mathbb{R}$, $x\in\mathbb{R}$ and $M>0$. Define the event (see Fig \ref{FigB}), 
\begin{equation}
  B^{[x]}_{n,L,M} \coloneqq   (A^{[x]}_{n, L,M} )^{c}\coloneqq\left\{\forall u\in\mathbb{T}_{\le n-1}:\varPhi\left(n-|u|,L-[x+S_u]\right)<M\right\} ,
\end{equation} 
where we recall that $\mathbb{T}_{\le n}=\{u\in\mathbb{T}: |u|\le n\}$.
Then for any $\epsilon>0$, there exists some constant $C$ depending only on $\epsilon$ such that for all $n\geq 1$,
\begin{equation}
\sup_{L,x\in\mathbb{R}}\sup_{M \ge \epsilon n }\P(\{\L^{[x]}_{n}(L,L+1)\geq e^{M}\}\cap B^{[x]}_{n,L,M-\epsilon n})\leq C\exp(-n^2).
\end{equation} 
\end{lemma}

The following lemma is a reformulation of~\cite[Theorem 1.4]{Zhang23}, see also \cite[Lemma~3.3 and (3.36)]{CH20}. Similarly, the bound is uniform in both the level $L$ and the size $M$. See~\S\ref{sec-LDP-proof-a} for a proof.

\begin{lemma}\label{lem-LDP-level-set-above-the-line}
Suppose that $n\geq 1$, $L\in\mathbb{R}$, $x\in\mathbb{R}$ and  $0\leq M\leq n\ln 2$. Define the event (see Fig \ref{FigB})
\begin{equation}\label{def-event-A}
  A^{[x]}_{n, L,M}\coloneqq  (B^{[x]}_{n, L,M})^{c}=\left\{\exists u\in\mathbb{T}_{\le n-1}:\varPhi\left(n-|u|,L-[x+S_u]\right)\geq M\right\} . 
\end{equation} 
Then for any $\epsilon>0$, there exists some constant $c$ depending only on $\epsilon$ such that  
\begin{equation}
  \sup_{L, x\in\mathbb{R}}\sup_{0\leq M\leq n\ln2}\P(\{\L^{[x]}_{n}(L,\infty)\leq e^{M}\}\cap A^{[x]}_{n,L,M+\epsilon n})\leq\exp(-cn^2). 
\end{equation}  
\end{lemma}

\begin{figure}[tbp]
  \includegraphics[width=0.65\linewidth]{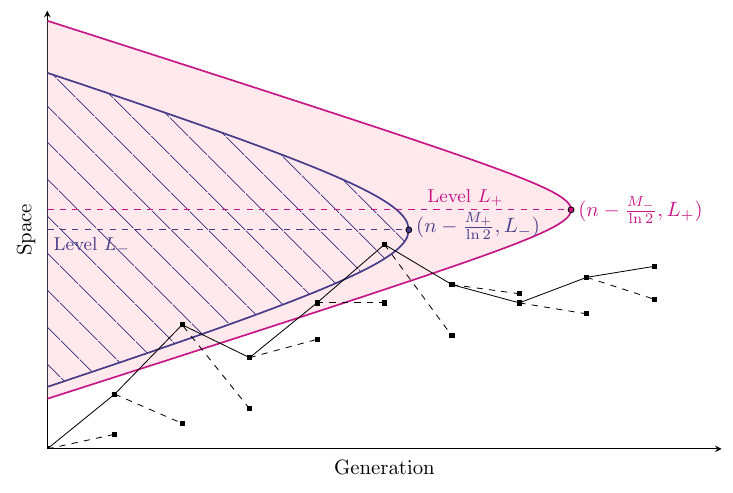}
  \caption
  {The event $A^{[0]}_{n,L+,M_-}$ (resp., $A^{[0]}_{n,L-,M_+}$) occurs when the BRW hits the pink (resp., shaded) region. Here $M_+>M_-$ and $L_+>L_-$.}
  \label{FigB}
\end{figure}

%%%%%%%%%%%%%%%%%%%%%%%%%%%%%%%%%%%%%%%%%%%%%%%%%%%
%%%%%%%%%%%%%%%%%%%%%%%%%%%%%%%%%%%%%%%%%%%%%%%%%%%
%%%%%%%%%%%%%% Subsection  %%%%%%%%%%%%%%%%%%%%%%%%
%%%%%%%%%%%%%%%%%%%%%%%%%%%%%%%%%%%%%%%%%%%%%%%%%%%
%%%%%%%%%%%%%%%%%%%%%%%%%%%%%%%%%%%%%%%%%%%%%%%%%%%

\subsection{Arbitrary-order polynomial decay}\label{sec-poly-decay}
We now prove that the right tail probability of $Z_{\infty}(\beta)_+ $ decays faster than any polynomial rate, uniformly in the initial position of the branching random walk. We first introduce the corresponding martingale when the branching random walk starts at an arbitrary position $x\in\mathbb{R}$. Define
\begin{equation}\label{eq-Z-x-n-1}
  Z^{[x]}_{n}(\beta)\coloneqq\frac{1}{2^{n}}\sum_{|u|=n}(\beta n-S_u-x)e^{\beta(S_u+x)-\frac{\beta^2}{2}n}=e^{\beta x}\left(Z_{n}(\beta)-xW_{n}(\beta)\right). 
\end{equation}
Taking the limit as $n\to\infty$, we set $Z^{[x]}_{\infty}(\beta)\coloneqq\lim_{n\to\infty}Z^{[x]}_{n}(\beta)=e^{\beta x}\left[Z_{\infty}(\beta)-xW_{\infty}(\beta)\right]$.

\begin{lemma}\label{lem-polynomial-decay}
  For any $K >0$,  there exists $C>0$ depending only on $\beta$ and $K$ such that for any $y\geq 1$,
  \begin{equation}
  \sup_{x\in\mathbb{R}}\P(Z_{\infty}^{[x]}(\beta)>y)\leq Cy^{-K}.
  \end{equation}
\end{lemma}

We first present the proof of Lemma~\ref{lem-polynomial-decay} assuming two auxiliary results, which will be proved afterwards.

\begin{lemma}\label{lem-Polynomial-decay-low-initial-position}
For any $\Delta>0$, there exist positive constants $c_{\eqref{eq-Polynomial-decay-low}}(\Delta), C_{\eqref{eq-Polynomial-decay-low}}(\Delta)$ depending only on $\Delta$ and $\beta$ such that for all $n\geq 1$,
  \begin{equation}\label{eq-Polynomial-decay-low}
  \sup_{x\in\mathbb{R}}\P(Z_{n}^{[x]}(\beta)\geq ne^{n\Delta})\leq C_{\eqref{eq-Polynomial-decay-low}}(\Delta)\exp\left\{-n^{2}c_{\eqref{eq-Polynomial-decay-low}}(\Delta)\right\}.
  \end{equation}
\end{lemma}

\begin{lemma}\label{lem-Polynomial-decay-high-initial-position}
For any $\theta>0$, there exist positive constants $c_{\eqref{eq-Polynomial-decay-high}}(\theta), C_{\eqref{eq-Polynomial-decay-high}}(\theta)$ depending only on $\theta$ and $\beta$ such that for all $n\geq 1$,
\begin{equation}\label{eq-Polynomial-decay-high}
  \P\left(Z_{n}(\beta)\geq\theta nW_{n}(\beta)\right)\leq C_{\eqref{eq-Polynomial-decay-high}}(\theta)\exp\left\{-n^{2}c_{\eqref{eq-Polynomial-decay-high}}(\theta)\right\}.  
\end{equation}
\end{lemma}

%%%%%%%%%%%%%%%%%%%%%%%%%%%%%%%%%%%%%%%%%%%%%%%%%%%
%%%%%%%%%%%%%%%%%%%%%%%%%%%%%%%%%%%%%%%%%%%%%%%%%%%
%%%Lemma Polynomial decay  %%%%%%%%%%%%%%%%%%%%%%%%
%%%%%%%%%%%%%%%%%%%%%%%%%%%%%%%%%%%%%%%%%%%%%%%%%%%
%%%%%%%%%%%%%%%%%%%%%%%%%%%%%%%%%%%%%%%%%%%%%%%%%%%

\begin{proof}[Proof of Lemma \ref{lem-polynomial-decay} admitting Lemmas~\ref{lem-Polynomial-decay-low-initial-position} and~\ref{lem-Polynomial-decay-high-initial-position}.]

Recall that both $W_n(\beta)$ and $Z_n(\beta)$ converge exponentially in $L^1$ to their respective limits $W_\infty(\beta)$ and $Z_\infty(\beta)$ (e.g.~\cite[Proposition~3.2]{CdRM24}). Consequently, there exist constants $C_0, c_0>0$ such that
\begin{equation}\label{eq-replace-Z-inf-by-Z-n}
  \E\left[|W_{\infty}(\beta)-W_{n}(\beta)|\right]+\E\left[|Z_{\infty}(\beta)-Z_{n}(\beta)|\right]\leq C_0e^{-c_{0}n},\quad\forall n\geq 1.
\end{equation}
Applying Markov's inequality and the expression of $|Z^{[x]}_{\infty}(\beta)-Z^{[x]}_{n}(\beta)|$ in~\eqref{eq-Z-x-n-1}, we obtain 
\begin{align}\label{eq-Z-x-tail-bound}
  \P(Z^{[x]}_{\infty}(\beta)>2y)&\leq y^{-1}\E\left[|Z^{[x]}_{\infty}(\beta)-Z^{[x]}_{n}(\beta)|\right]+\P\left(Z^{[x]}_{n}(\beta)>y\right)\\
  &\leq C_{0}y^{-1}(1+|x|)e^{\beta x}e^{-c_{0}n}+\P\left(Z_{n}^{[x]}(\beta)>y\right)
\end{align} 
for all $n\geq 1$. Given $K>0$, we choose a sufficiently large constant $A$ such that
\begin{equation}
  A\beta\geq K\quad\text{and}\quad  A\frac{3\beta}{c_0}c_{(\ref{eq-Polynomial-decay-high})}(\frac{c_0}{4\beta})\geq K.
\end{equation}
We now divide the proof into two cases.

\smallskip
\noindent
\underline{\textit{Case 1.}}  
Suppose  $x\in(-\infty, A\ln y]$ and apply inequality~\eqref{eq-Z-x-tail-bound} with $n=\lceil\frac{\beta A+2 K}{c_0}\ln y\rceil$. For large $y$, by the property of $x\mapsto |x|e^{\beta x}$ and our choice of $n$ and $A$, we have 
\begin{equation}
  C_{0}(1+|x|)e^{\beta x}e^{-c_{0}n}y^{-1}\lesssim y^{-1}(1+A\ln y)e^{-2K\ln y}\leq e^{-K\ln y}.
\end{equation}
Now set $\Delta=\frac{c_0}{\beta A+4K}$ so that $y \geq n e^{n\Delta}$ for large $y$. It follows from Lemma~\ref{lem-Polynomial-decay-low-initial-position} that 
\begin{equation}
  \P(  Z_{n}^{[x]}(\beta)> y )  \leq  \P(  Z_{n}^{[x]}(\beta) \geq n e^{\Delta n } ) \lesssim e^{ - n^2 c_{\eqref{eq-Polynomial-decay-low}}(\Delta) } \leq e^{-K \ln y} .
\end{equation} 
Together with inequality~\eqref{eq-Z-x-tail-bound}, this implies the desired result.

\smallskip
\noindent
\underline{\textit{Case 2.}} Suppose $x>A\ln y$. We now choose $n=\lceil\frac{3\beta}{c_0}x\rceil$ so that
\begin{equation}
  C_0(1+|x|)e^{\beta x}e^{-c_{0}n}y^{-1}\leq C_{0}(1+|x|)e^{-2\beta x}\leq e^{-K\ln y}
\end{equation}
provided that $y$ is sufficiently large. {Moreover, for $y$ sufficiently large, this choice of $n$ gives $x\geq c_0 n/(4\beta)$.} Applying Lemma~\ref{lem-Polynomial-decay-high-initial-position} and by our choice of $A$, we get 
\begin{align}
  \P\left(Z_{n}^{[x]}(\beta)>y\right)&\leq\P\left(Z_{n}(\beta)-xW_{n}(\beta)>0\right)\le\P\Bigl(Z_{n}(\beta)\geq\frac{c_{0}n}{4 \beta}W_{n}(\beta)\Bigr)\\
  &\lesssim
  \exp\Bigl(-n^2 c_{(\ref{eq-Polynomial-decay-high})}(\frac{c_0}{4 \beta})\Bigr)\leq e^{-K\ln y}. 
\end{align} 
The desired result then follows from inequality~\eqref{eq-Z-x-tail-bound}.
\end{proof}

%%%%%%%%%%%%%%%%%%%%%%%%%%%%%%%%%%%%%%%%%%%%%%%%%%%
%%%%%%%%%%%%%%%%%%%%%%%%%%%%%%%%%%%%%%%%%%%%%%%%%%%
%%%Lemma Polynomial-decay-low-initial-position %%%
%%%%%%%%%%%%%%%%%%%%%%%%%%%%%%%%%%%%%%%%%%%%%%%%%%%
%%%%%%%%%%%%%%%%%%%%%%%%%%%%%%%%%%%%%%%%%%%%%%%%%%%

The rest of this subsection is devoted to proving Lemmas~\ref{lem-Polynomial-decay-low-initial-position} and~\ref{lem-Polynomial-decay-high-initial-position}. Before stating the proofs, we record a deterministic comparison of level-set events. It will be used later to rule out the simultaneous occurrence of a large level set at level $L$ and a relatively small level set at the shifted level $L+h$.

\begin{lemma} 
\label{lem-level-set-event-comparison}
Fix $n\geq 1$, $x,L\in\mathbb{R}$, $h>0$, and
$M_{\mathrm{A}},M_{\mathrm{B}}\in\mathbb{R}$ with
$0<M_{\mathrm{A}}\leq n\ln 2$. Then the non-intersection relation
\begin{equation}\label{eq-level-set-comparison-disjoint}
 A^{[x]}_{n,L,M_{\mathrm{A}}}
 \cap
 B^{[x]}_{n,L+h,M_{\mathrm{B}}}
 =\emptyset
\end{equation}
holds, provided that there exists $q\in(0,\beta_c)$ such that
\begin{subequations}\label{eq-level-set-comparison-conditions}
\begin{align}
  M_{\mathrm{A}} -  qh & \ge  M_{\mathrm{B}},
 \label{eq-level-set-comparison-height}\\
 \frac{M_{\mathrm{A}}}{\ln 2} & \geq \frac{h}{2q},
 \label{eq-level-set-comparison-left-endpoint}\\
 [2\ln2-q^2]n^2 & -2 M_{\mathrm{A}}n + qnh
 \leq\frac{h^2}{4}. 
 \label{eq-level-set-comparison-right-endpoint}
\end{align}
\end{subequations}
\end{lemma}

\begin{proof}[Proof of Lemma~\ref{lem-Polynomial-decay-low-initial-position}]
For any measurable set $A\subset\mathbb{R}$, define 
\begin{equation}
  Z^{[x]}_{n}(\beta, A)\coloneqq\frac{e^{n\beta^2/2}}{2^{n}}\sum_{|u|=n}\mathbf{1}_{\{x+S_u\in A\}}[\beta n-(S_u+x)]e^{\beta[(S_u+x)-\beta n]}.
\end{equation}
In particular $Z^{[x]}_{n}(\beta)=Z^{[x]}_{n}(\beta,\mathbb{R})$. Let $I_{n}\coloneqq(\frac{\beta}{2}n,\infty)$. Since the function $y\mapsto ye^{-\beta y}$ is decreasing in $(1/\beta,\infty)$,  we have  for large $n$,
\begin{equation}
  Z^{[x]}_n\left(\beta, I_{n}^{c}\right)\leq e^{\frac{\beta^2}{2}n}\sup_{y\geq\beta n/2}ye^{-\beta y}\leq\frac{\beta n}{2}=o(e^{\Delta n}).  
\end{equation}
It suffices to show that for any fixed $\Delta$ > 0, there exist $c,C > 0$
depending only on $\Delta,\beta$   such that 
\begin{equation}
 \P \biggl( Z_n^{[x]}(\beta,I_n)\geq \beta n e^{\Delta n } /2 \biggr) \le Ce^{-cn^2} . 
\end{equation} 
 
\smallskip 
\noindent
\underline{\textit{Step 1.}} 
Define $I_{n,k}\coloneqq(\beta n-k, \beta n-k+1]$ for $1\leq k\leq\frac{\beta}{2}n$ and let $J_{n}\coloneqq(\beta n, \infty)$. We assume that $\beta n/2$ is an integer up to minor changes in the sequel using $\lceil\beta n/2\rceil$. We decompose  $Z^{[x]}_n \left( \beta,  I_{n}  \right)$ as follows:
\begin{equation}
  Z^{[x]}_n\left(\beta, I_{n}\right)=\sum_{k=1}^{\beta n/2}Z^{[x]}_n\left(\beta,I_{n,k}\right)-| Z^{[x]}_n\left(\beta, J_{n}\right)|. 
\end{equation}
Define the events $  E^{[x]}_{n,k} \coloneqq \left\{  Z^{[x]}_n \left( \beta,  I_{n,k}  \right) - \frac{2}{\beta n} | Z^{[x]}_n \left( \beta,  J_{n}\right)|  \geq  e^{\Delta n}   \right\}$. The pigeonhole principle gives
\begin{equation}
 \left\{  Z^{[x]}_n \left( \beta,  I_{n}  \right) \geq \frac{\beta n}{2} e^{\Delta n} \right\} \subset \bigcup_{k=1}^{ {\beta n}/{2} }   E^{[x]}_{n,k}  \ . 
\end{equation} 
By using the union bound, it is sufficient to show that
\begin{equation}\label{eq-poly-stp1}
  \sup_{x \in \mathbb{R}} \sup_{1 \leq k \leq \beta n/2} \P \left(  E^{[x]}_{n,k} \right) \leq e^{-\Theta(n^2)} . 
\end{equation}
\smallskip
 
\noindent
    \underline{\textit{Step 2.}} For any $1 \leq k \leq \beta n/2$,  we have the following  upper bound of $ Z^{[x]}_n \left( \beta,  I_{n,k}  \right) $:
\begin{align}
  Z^{[x]}_n \left( \beta,  I_{n,k}  \right)  & \leq  e^{-\Phi(\beta) n }\sum_{|u|=n} \ind{ \beta n-(x+S_u) \in [k-1,k]}  [\beta n-(x+S_{u})]  e^{-\beta [\beta n - (x+S_{u} )] }  \\
  & \leq  e^{-\Phi(\beta) n } \, e^{\beta} k e^{-\beta k}  \L_{n}^{[x]}(I_{n,k}). 
\end{align} 
Similarly, we see that for any $\ell\ge1$,
\begin{align}
  |Z^{[x]}_n \left( \beta,  J_{n}  \right) |  &\geq   e^{-\Phi(\beta) n } \sum_{|u|=n} \ind{ (x+S_u) -\beta n \geq \ell}  (x+S_{u}- \beta n) e^{\beta (x+S_{u}- \beta n)} \\
  &\geq   e^{-\Phi(\beta) n } \, \ell e^{\beta \ell}     \L_{n}^{[x]}( \beta n + \ell, \infty) . 
\end{align}
Therefore, for all  $1 \leq k \leq \beta n /2$  and  $\ell \geq 1$, we have 
\begin{equation}
  \label{eq:E-n-k-E-n-k-l}
 E^{[x]}_{n,k} \subset   E^{[x]}_{n,k,\ell}  \coloneqq   \left\{ e^{\beta} k e^{-\beta k}  \L_{n}^{[x]}(I_{n,k})  -  \frac{2\ell}{\beta n} e^{\beta \ell}     \L_{n}^{[x]}( \beta n + \ell, \infty)  \geq  e^{ ( \Phi(\beta) + \Delta )n }   \right\}.
\end{equation}
In order to prove \eqref{eq-poly-stp1}, we only need to control the probability $\P( E^{[x]}_{n,k,\ell} )$ with a well-chosen $\ell$.

\smallskip 
\noindent
    \underline{\textit{Step 3.}} 
 We choose $\epsilon$ sufficiently small and set $\Delta' \coloneqq \frac{1}{2}\Delta-\epsilon$, $\delta\coloneqq 2 \sqrt{\epsilon}$ such that
\begin{equation}\label{choice-epsilon-delta}
  0<\epsilon < \frac{1}{4}  \Delta \ , \  \delta < \frac{\beta}{2} \ , \ \beta\delta>\epsilon \ , \     3 \beta  \delta \leq  \Delta'  \ , \          \frac{3\delta}{  \beta} \leq  \frac{\Delta'}{ \ln 2} .
\end{equation} 
For fixed $1 \le k \le \frac{\beta n}{2}$, choose
\begin{equation}\label{eq-low-choice-ell}
 \ell\coloneqq\max\{1,\lceil\delta n-k\rceil\},\qquad h\coloneqq k+\ell. 
\end{equation}
 For all sufficiently large $n$, this choice satisfies
\begin{equation}\label{eq-low-choice-ell-bounds}
 1\leq\ell\leq\delta n,\qquad h\geq\delta n.
\end{equation}
It follows from the definition that when  $E^{[x]}_{n,k,\ell}$ occurs,    for all sufficiently large $n$, we have 
\begin{equation}\label{eq-low-level-set-range}
r_{n,k}\coloneqq\beta k+\Phi(\beta)n+\frac{ \Delta}{2}n\leq \ln \L_n^{[x]}(I_{n,k})\leq n\ln2.
\end{equation}
If $r_{n,k}>n\ln2$, the event $E^{[x]}_{n,k,\ell}$ is empty. Otherwise, define
\[
 N_{n,k}\coloneqq \max\left\{1,\left\lceil\frac{2}{\epsilon}\Big(\ln2-\frac{r_{n,k}}{n}\Big)\right\rceil\right\},\qquad
 \mathcal A_{\epsilon,n,k}\coloneqq\left\{\frac{r_{n,k}}{n}+\frac{j\epsilon}{2}:0\leq j<N_{n,k}\right\}.
\]
 As the intervals $[an,(a+\epsilon/2)n]$, $a\in\mathcal A_{\epsilon,n,k}$, cover $[r_{n,k},n\ln2]$,  applying the union bound yields
\begin{equation}
  \label{eq-adding-constraint-0}
 \P(  E^{[x]}_{n,k,\ell}  ) \leq \sum_{a \in \mathcal{A}_{\epsilon,n,k}} \P( E^{[x]}_{n,k,\ell} , \,   \L_{n}^{[x]}( I_{n,k} ) \in [e^{ a n }, e^{(a+\epsilon/2) n}]  ). 
\end{equation} 

\smallskip 
\noindent
 \underline{\textit{Step 4.}} In this step, we will combine \eqref{eq-adding-constraint-0} with Lemmas~\ref{lem-LDP-level-set-below-the-line} and~\ref{lem-LDP-level-set-above-the-line} to show that
\begin{equation}
  \P(  E^{[x]}_{n,k,\ell}  ) \leq  \sum_{a \in \mathcal{A}_{\epsilon,n,k}}  \P\left( A^{[x]}_{n,\beta n- k, a n- \epsilon n} \cap  B^{[x]}_{n,\beta n + \ell,  (a + 2\epsilon) n -  \beta h  }  \right)  +  2C  | \mathcal{A}_{\epsilon,n,k}| e^{-c n^2}. \label{eq:union-bound-E-nkl}
\end{equation}
 
To this end, note that on the event $E^{[x]}_{n,k,\ell} \cap \{\L_{n}^{[x]}( I_{n,k})\in[e^{an}, e^{(a+\epsilon/2) n}] \} $, we have
\begin{equation}\label{eq-constraint-1}
  \L_{n}^{[x]}( \beta n + \ell, \infty)   \leq \frac{\beta n}{2}e^{\beta}  \frac{k}{\ell} e^{-\beta  h}   e^{(a+\epsilon/2)n} \leq e^{-\beta h}   e^{(a+\epsilon)n} . 
\end{equation} 
 for  sufficiently large $n$. This imposes an additional constraint on $\L_{n}^{[x]}(\beta n + \ell, \infty)$, which will turn out to be negligible in probability.
Indeed,  for every $a\in\mathcal A_{\epsilon,n,k}$, the definition of the grid and~\eqref{choice-epsilon-delta} give
\begin{equation}\label{eq-low-LDP-parameter-ranges}
 \epsilon n\leq an,
 \qquad
 0\leq (a+\epsilon)n-\beta h\leq n\ln2. 
\end{equation}
The first inequality follows from $an\geq r_{n,k}\geq \Delta n/2>\epsilon n$. For the second pair, using $\ell\leq\delta n$ and $h\geq\delta n$, we obtain
  $ (a+\epsilon)n-\beta h
  \geq n[\Phi(\beta)+\frac{\Delta}{2}+\epsilon-\beta\delta]>0,  $ 
  and  $
 (a+\epsilon)n-\beta h
 \leq n (\ln2+\epsilon-\beta\delta) \leq n\ln2.  $
 Thus all the size parameters used below lie in the ranges required by Lemmas~\ref{lem-LDP-level-set-below-the-line} and~\ref{lem-LDP-level-set-above-the-line}: 

 According to Lemma~\ref{lem-LDP-level-set-below-the-line}  there exists constant $C>0$ depending only on $\epsilon$ such that 
 \begin{equation}
  \label{eq-adding-constraint-1}
\sup_{x \in \mathbb{R}} \P \left(  E^{[x]}_{n,k,\ell},  \L_{n}^{[x]}( I_{n,k} )  \geq e^{ a n },  B^{[x]}_{n,\beta n- k, a n- \epsilon n} \right) \leq C e^{-n^2} .
 \end{equation}
On the other hand, combining \eqref{eq-constraint-1} with Lemma \ref{lem-LDP-level-set-above-the-line} implies that there are constants $c$, $C$ (depending only on $\epsilon$) such that 
 \begin{align}
  & \sup_{x \in \mathbb{R}} \P \left( E^{[x]}_{n,k,\ell},  \L_{n}^{[x]}( I_{n,k} )  \in  [e^{ a n }, e^{(a+\epsilon/2) n}],     A^{[x]}_{n,\beta n + \ell,  (a + 2\epsilon) n -  \beta h } \right)  \\
  & \quad \leq   \sup_{x \in \mathbb{R}}  \P \left(  \L_{n}^{[x]}( \beta n + \ell, \infty) \leq e^{ (a+\epsilon)n - \beta h },  A^{[x]}_{n,\beta n + \ell,  (a + 2\epsilon) n -  \beta h } \right) \leq C e^{- c n^2} . \label{eq-adding-constraint-2}
 \end{align}
Plugging \eqref{eq-adding-constraint-1} and \eqref{eq-adding-constraint-2} back into \eqref{eq-adding-constraint-0},  the desired result \eqref{eq:union-bound-E-nkl} follows.

 \smallskip 
\noindent
\underline{\textit{Step 5.}} 
Fix $1\leq k\leq\beta n/2$ and let $\ell$ and $h$ be chosen as in~\eqref{eq-low-choice-ell}--\eqref{eq-low-choice-ell-bounds}.
We claim that for all $a\in\mathcal{A}_{\epsilon,n,k}$,
\begin{equation}\label{Above-Below-Contradiction}
 A^{[x]}_{n,\beta n-k,an-\epsilon n}
 \cap
 B^{[x]}_{n,\beta n+\ell,(a+2\epsilon)n-\beta h}
 =\emptyset.
\end{equation}
 Substituting~\eqref{Above-Below-Contradiction} into~\eqref{eq:union-bound-E-nkl} gives 
\[ 
 \P(E^{[x]}_{n,k})
 \leq\P(E^{[x]}_{n,k,\ell})
 \leq2C|\mathcal{A}_{\epsilon,n,k}|e^{-cn^2}   \le C'e^{-c n^2}. 
\]
 with constant $C'$ depending only on $\epsilon$. This implies \eqref{eq-poly-stp1} and completes the proof. 

 It remains to prove \eqref{Above-Below-Contradiction}.
To this end, we apply Lemma~\ref{lem-level-set-event-comparison} with
\[
 L=\beta n-k,\quad  h=k+\ell,\quad  q=\beta-\delta,\quad  M_{\mathrm{A}}=(a-\epsilon)n,\quad  M_{\mathrm{B}}=(a+2\epsilon)n-\beta h.
\] 
By~\eqref{choice-epsilon-delta} since $0<\delta<\beta/2$,  we have $\beta/2 < q < \beta < \beta_{c}$. 
For $a\in\mathcal{A}_{\epsilon,n,k}$, the definition of
$\mathcal{A}_{\epsilon,n,k}$  gives  
\begin{equation}\label{eq-low-Ma-bounds}
\beta k+\Phi(\beta)n+\Delta' n
 \leq M_{\mathrm{A}}
 \leq n\ln2-\epsilon n.
\end{equation}
In particular, $0<M_{\mathrm{A}}\leq n\ln2$.
It remains to verify the three
inequalities in Lemma~\ref{lem-level-set-event-comparison}.

\begin{itemize}
\item We first verify~\eqref{eq-level-set-comparison-height}. From the definitions of $M_{\mathrm{A}}$, $M_{\mathrm{B}}$, and $q$, we have
\begin{align*}
 M_{\mathrm{A}}-qh-M_{\mathrm{B}}= (a-\epsilon)n-(\beta-\delta)h -\bigl[(a+2\epsilon)n-\beta h\bigr] =\delta h-3\epsilon n.
\end{align*}
 By~\eqref{eq-low-choice-ell-bounds}, $h\geq\delta n$.  Recalling that
$\delta=2\sqrt{\epsilon}$, we obtain
$
 M_{\mathrm{A}}-qh-M_{\mathrm{B}}
 \geq\delta^2n-3\epsilon n 
 =\epsilon n>0$.
Therefore $M_{\mathrm{B}}\leq M_{\mathrm{A}}-qh$, as required.

\item Secondly, we verify    \eqref{eq-level-set-comparison-left-endpoint}. Since
$k\leq\beta n/2$, $\ell\leq\delta n$, and $\delta<\beta/2$, we have
$
 h=k+\ell
 \leq\beta n$.
Moreover, $q=\beta-\delta\geq\beta/2$. Using the identity
$
 \frac{1}{q}=\frac{1}{\beta}+\frac{\delta}{\beta q}$,
we obtain
\begin{align} 
 \frac{h}{2q}=\frac{k}{2\beta}+\frac{\ell}{2\beta}  +\frac{\delta h}{2\beta q} \leq\frac{k}{2\beta}+\frac{\delta}{2\beta}n   +\frac{\delta}{\beta}n  \leq\frac{k}{2\beta}+\frac{\Delta'}{\ln2}n. 
\label{eq-low-left-endpoint-bound}
\end{align}
Here we have used the last condition in
\eqref{choice-epsilon-delta}, namely
$\frac{ 3\delta}{2\beta}\leq \frac{\Delta'}{\ln2}$.  
Using the lower bound in
\eqref{eq-low-Ma-bounds} and the fact
$\Phi(\beta)=\ln2-\beta^2/2$, we get
\begin{equation}
 \frac{M_{\mathrm{A}}}{\ln2}
 \geq \frac{\beta}{\ln2}k
   +\Big(1-\frac{\beta^2}{2\ln2}\Big)n
   +\frac{\Delta'}{\ln2}n .
\end{equation}
Thus, it suffices to verify 
\[
 g(k)\coloneqq
 \Big( 1-\frac{\beta^2}{2\ln2}\Big)n
 \mathbin{+}\Big(\frac{\beta}{\ln2}-\frac{1}{2\beta}\Big)k \ge 0
 \qquad\text{for all } 0\leq k\leq\frac{\beta n}{2}.
\]
Since $g$ is linear in $k$,   
$  g(0)
 = (1-\frac{\beta^2}{2\ln2})n >0$ and  $g(\frac{\beta n}{2})
 =\frac{3n}{4}>0$, the required condition follows.  

\item Finally, we verify~\eqref{eq-level-set-comparison-right-endpoint}.  
% It is enough to show  
% \begin{equation}
%  [2\ln2-q^2]n^2-2M_{\mathrm{a}}n+qnh
%  \leq 0.
% \end{equation}
By the lower bound in
\eqref{eq-low-Ma-bounds}, the definition $h=k+\ell$,  we can bound the left-hand side in \eqref{eq-level-set-comparison-right-endpoint} from above by 
\begin{align*}
[2\ln2-q^2]n^2 -2\bigl[\beta k+\Phi(\beta)n+\Delta' n\bigr]n +qn(k+\ell)
&=[\beta^2-q^2]n^2-(2\beta-q)kn \\
&\quad +q\ell n-2\Delta' n^2.
\end{align*}
Since $2\beta-q=\beta+\delta>0$ and $k\geq1$, the second term in the last line is non-positive. Furthermore, $ \beta^2-q^2
 =\delta(\beta+q)\leq2\beta\delta $ and $ 
 q\ell n\leq\beta\delta n^2$. 
It follows that
\begin{align*}
 [2\ln2-q^2]n^2-2M_{\mathrm{A}}n+qnh   \le (3\beta\delta-2\Delta')n^2 \leq-\Delta' n^2<0,
\end{align*}
where the last inequality is exactly the condition $3\beta\delta\leq \Delta'$
in~\eqref{choice-epsilon-delta}. 
\end{itemize}
All the three conditions of Lemma~\ref{lem-level-set-event-comparison} are therefore satisfied. This completes the proof.  
\end{proof}

%%%%%%%%%%%%%%%%%%%%%%%%%%%%%%%%%%%%%%%%%%%%%%%%%%%
%%%%%%%%%%%%%%%%%%%%%%%%%%%%%%%%%%%%%%%%%%%%%%%%%%%
%%%Lemma Polynomial-decay-high-initial-position %%%
%%%%%%%%%%%%%%%%%%%%%%%%%%%%%%%%%%%%%%%%%%%%%%%%%%%
%%%%%%%%%%%%%%%%%%%%%%%%%%%%%%%%%%%%%%%%%%%%%%%%%%%

\begin{proof}[Proof of Lemma~\ref{lem-Polynomial-decay-high-initial-position}]
Since $W_n(\beta)\geq0$, the event in~\eqref{eq-Polynomial-decay-high} is decreasing in $\theta$. It therefore suffices to prove the result for sufficiently small $\theta$. 

\smallskip
\noindent
\underline{\textit{Step 1.}} Fix $\theta\in(0,\beta/2)$ and set $ \tilde\beta\coloneqq\beta-\theta\in(\beta/2,\beta)$.
Choose $\delta=\theta/2$ then $\epsilon\in(0,1)$ sufficiently small,
\begin{equation}\label{eq-choosing-eps-del-2}
 \sqrt\epsilon<\beta,\quad
 \delta\geq 4 \epsilon^{1/4},\quad
 \beta\delta>6\epsilon,\quad
 3\epsilon+\frac{\sqrt{\epsilon}\,  \ln2}{\beta(\beta-\sqrt\epsilon)}  
 <\Phi(\beta)\wedge\frac{\ln2}{2}.
\end{equation} 
Define
\[
 v\coloneqq \tilde\beta+\delta \in(0,\beta) \ , \  \lambda\coloneqq\frac{v^2}{2\beta}\in(0,\beta/2) \ ; \ 
 \text{ and } 
 E_n\coloneqq \Big\{\L_n(vn,\infty)\geq
 e^{[\Phi(v)-\epsilon] n} \Big\}.
\] 

We first show that $E_n$ occurs with high  probability.
Apply Lemma~\ref{lem-LDP-level-set-above-the-line} with error parameter $\epsilon/2$, $L=(\tilde{\beta}+\delta)n$, $M=\Phi(\tilde{\beta}+\delta) n - \epsilon n$ and $x=0$. Since  the root vertex $\rho$ satisfies  $\varPhi \left(n-|\rho|, L - S_{\rho}\right) = \Phi(\tilde{\beta}+\delta)n$, the event $A^{[0]}_{ n, L,M + \epsilon n/2}$ occurs surely. Hence there exist constants $c, C$ depending only on $ \beta,\theta$ such that 
  \begin{equation}
    \label{eq: E-n-good}
  \P( E_{n}^{c} ) \leq C e^{- c n^2 } \, ,  \text{ for all }  n\geq 1.
  \end{equation}

\smallskip
\noindent
\underline{\textit{Step 2.}}
We next show that  on 
$\{Z_n(\beta)\geq\theta nW_n(\beta)\}\cap E_n$ it holds  
\begin{equation}
  \sum_{|u|=n}\ind{0\leq\tilde\beta n-S_u<\lambda n}
 (\tilde\beta n-S_u)e^{\beta(S_u-\tilde\beta n)} \geq\frac{\delta n}{2}e^{\beta\delta n}
 \L_n(vn,\infty).\label{eq-bnd-0-lambda-n}
\end{equation}

Observe that the event $\{
   Z_{n}(\beta) \ge  \theta n W_{n}(\beta) \}  $ occurs if and only if 
\begin{equation}\label{eq-bound-ab-bl-0} 
   \sum_{|u|=n}  \ind{\tilde{\beta} n - S_u  \ge 0 }   (\tilde{\beta} n - S_u)  e^{\beta(S_u- \tilde{\beta} n   )}   \geq \sum_{|u|=n} \ind{S_u- \tilde{\beta} n \geq 0}(S_u-\tilde{\beta} n ) e^{\beta(S_u- \tilde{\beta} n   )}  .
\end{equation} 
On $E_n$, the right-hand side of \eqref{eq-bound-ab-bl-0} is bounded from below by
\begin{equation}\label{eq-bound-ab-1}
 \sum_{|u|=n}\ind{S_u\geq vn}(S_u-\tilde\beta n)
 e^{\beta(S_u-\tilde\beta n)}
 \geq\delta n \, e^{\beta\delta n} \,  \L_n(vn,\infty).
\end{equation}
Moreover, since $y\mapsto ye^{-\beta y}$ is decreasing for large $y$ and $\beta\lambda=v^2/2$, the contribution to the left-hand side of \eqref{eq-bound-ab-bl-0} from particles satisfying $ \tilde\beta n-S_u\geq\lambda n$ is bounded from above by
\begin{equation}\label{eq-bound-bl-2}
 \sum_{|u|=n}\ind{\tilde\beta n-S_u\geq\lambda n}
 (\tilde\beta n-S_u)e^{\beta(S_u-\tilde\beta n)}
 \leq\beta n e^{\Phi(v)n} .
\end{equation}  Since $\beta\delta>\epsilon$, the right-hand side of~\eqref{eq-bound-bl-2} is, for all sufficiently large $n$, at most one half of the right-hand side of~\eqref{eq-bound-ab-1}. Combining this with~\eqref{eq-bound-ab-bl-0}, we obtain \eqref{eq-bnd-0-lambda-n}.

\smallskip
\noindent
\underline{\textit{Step 3.}}
For $1\leq k\leq\lambda n$, define
\[
 \tilde I_{n,k}\coloneqq(\tilde\beta n-k,\tilde\beta n-k+1].
\]
In this step we prove that  $\P\big(\{Z_n(\beta)\geq\theta nW_n(\beta)\}\cap E_n\big) $ is bounded  above by 
\begin{equation} 
 \sum_{k=1}^{\lambda n}\sum_{b\in\mathcal B}
 \P \bigl(  
 \L_n(vn,\infty)\leq e^{bn},\ 
 \L_n(\tilde I_{n,k})\geq
 e^{bn+\beta(k+\delta n)-2\epsilon n}\bigr).
 \label{eq-high-union-bound}
\end{equation}
where  
\[
 \mathcal B\coloneqq
 \{\Phi(v)+j\epsilon:0\leq j<J\}\cup\{\ln2\} \  , \ \text{ and  } \  J\coloneqq\left\lceil\frac{\ln2-\Phi(v)}{\epsilon}\right\rceil .
\]
%As in the previous proof, we assume for convenience that $\lambda n$ is an integer; the general case only changes the last interval. 

Indeed, on $E_n$, there exists $b\in\mathcal B$ such that
$
 \L_n(vn,\infty)\in[e^{(b-\epsilon)n},e^{bn}]$. Applying the pigeonhole principle to \eqref{eq-bnd-0-lambda-n} yields that   there exists $1\leq k\leq\lambda n$ satisfying
\begin{align}
 \frac{\delta}{2\lambda}e^{\beta\delta n}\L_n(vn,\infty)\leq\sum_{|u|=n}\ind{S_u\in\tilde I_{n,k}} (\tilde\beta n-S_u)e^{\beta(S_u-\tilde\beta n)} \leq e^\beta k e^{-\beta k}\L_n(\tilde I_{n,k}).
 \label{eq-constraint-01}
\end{align} Consequently, 
$
 \L_n(\tilde I_{n,k})\geq
 e^{bn+\beta(k+\delta n)-2\epsilon n}
$, since 
 uniformly in $1\leq k\leq\lambda n$, the polynomial prefactor in~\eqref{eq-constraint-01} can be absorbed into $e^{-\epsilon n}$ for   large $n$.  This proves \eqref{eq-high-union-bound}.

\smallskip
\noindent
\underline{\textit{Step 4.}} Combining \eqref{eq: E-n-good} with 
\eqref{eq-high-union-bound} 
and using that \(|\mathcal B|=O_{\epsilon}(1)\), it suffices to show that there exist constants $0<c,C<\infty$ such that for any   $1 \le k \le \lambda n$ and any $b \in \mathcal{B}$, 
 \begin{equation}
  \label{eq-high-union-bound-k-b}
  \P \bigl(  
 \L_n(vn,\infty)\leq e^{bn},\ 
 \L_n(\tilde I_{n,k})\geq
 e^{bn+\beta(k+\delta n)-2\epsilon n}\bigr) \le C e^{-c n^2}
 \end{equation}

 Define
\begin{equation}\label{eq-high-comparison-parameters}
 h\coloneqq k+\delta n,\qquad
 M_{\mathrm{A}}\coloneqq bn+\beta h-3\epsilon n,\qquad
 M_{\mathrm{B}}\coloneqq bn+\epsilon n.
\end{equation}
Since $b\geq\Phi(v)>\Phi(\beta)>3\epsilon$, we have $M_{\mathrm{A}}+\epsilon n=bn+\beta h-2\epsilon n\geq\epsilon n$. Moreover, $0\leq bn\leq n\ln2$ by the definition of $\mathcal B$. Lemmas~\ref{lem-LDP-level-set-below-the-line} and~\ref{lem-LDP-level-set-above-the-line} now give
\begin{equation}
  \label{eq-divide-ac-A-B}
  \P \Bigl(   
 \L_n(\tilde I_{n,k})\geq e^{M_{\mathrm{A}}+\epsilon n},
 B^{[0]}_{n,\tilde\beta n-k,M_{\mathrm{A}}}\Bigr)
  + \P \Bigl(  
 \L_n(vn,\infty)\leq e^{bn},
 A^{[0]}_{n,vn,M_{\mathrm{B}}}
  \Bigr) \leq Ce^{-cn^2}. 
\end{equation} 
We further claim that  
\begin{equation}
\label{eq-A-B-not-intersect}
   A^{[0]}_{n,\tilde\beta n-k,M_{\mathrm{A}}}
   \cap B^{[0]}_{n,vn,M_{\mathrm{B}}}=\emptyset.
\end{equation}
It then follows from \eqref{eq-divide-ac-A-B}, \eqref{eq-A-B-not-intersect}, 
and the definitions in \eqref{eq-high-comparison-parameters} that 
\eqref{eq-high-union-bound-k-b} holds.

\smallskip
\noindent
\underline{\textit{Step 5.}}
We finish the proof by showing  \eqref{eq-A-B-not-intersect}. We apply Lemma~\ref{lem-level-set-event-comparison} with $
 h, M_{\mathrm{A}},M_{\mathrm{B}}$ as in~\eqref{eq-high-comparison-parameters} and 
\[
 L=\tilde\beta n-k,\qquad 
 q=\beta- \sqrt{\epsilon}
\]
Notice that $L+h=vn$ and $q\in(0,\beta_c)$. Moreover,
$0<M_{\mathrm{A}}\leq n\ln2$ for every nonempty term retained above. We verify the three conditions of Lemma~\ref{lem-level-set-event-comparison}.

\begin{itemize}
\item Since $\beta-q=\sqrt{\epsilon}$ and $h\geq\delta n$, 
\[ M_{\mathrm{A}}-qh-M_{\mathrm{B}}
  =(\beta-q)h-4\epsilon n \geq(\delta\sqrt{\epsilon}-4\epsilon)n\geq0,\] 
where the last inequality follows from $\delta\geq4\epsilon^{1/4}$ and $\epsilon<1$. This proves~\eqref{eq-level-set-comparison-height}.

\item We next prove~\eqref{eq-level-set-comparison-left-endpoint}. Since $M_{\mathrm{A}} \ge \Phi(v) n + \beta h - 3 \epsilon n$, and $ \frac{\ln2}{2q}
 =\frac{\ln2}{2\beta}
 +\frac{ \sqrt{\epsilon} \ln2}{2\beta q}$,    it is enough to show 
\begin{equation}\label{eq-high-left-sufficient}
 \Phi(v)n+\Big(\beta-\frac{\ln2}{2\beta}\Big)h
 \geq3\epsilon n+\frac{  \sqrt{\epsilon} \ln2}{2\beta q}h.
\end{equation}  
Note that $h\leq(\lambda+\delta)n\leq\beta n \le 2 n$. 
If $\beta-\frac{\ln2}{2\beta}\geq0$, then by~\eqref{eq-choosing-eps-del-2},
\[
 3\epsilon n+\frac{ \sqrt{\epsilon}\ln2}{2\beta q}h
 \leq \Bigl[  3\epsilon+\frac{  \sqrt{\epsilon} \ln2}{ \beta q}\Bigr]  n
 \leq\Phi(\beta)n\leq\Phi(v)n.
\]
If $\beta-\frac{\ln2}{2\beta}<0$, the left-hand side of~\eqref{eq-high-left-sufficient} is minimized at $h=(\lambda+\delta)n$. Using $\beta\lambda=v^2/2$, $v<\beta$,  $\delta = \theta/2 < \beta/2$, and \eqref{eq-choosing-eps-del-2}, we obtain
\begin{align*}
 \Phi(v)n+ \Bigl(    \beta-\frac{\ln2}{2\beta}\Bigr) (\lambda+\delta)n \geq n\ln2\Big[1-\frac{v^2}{4\beta^2}-\frac{\delta}{2\beta}\Big] \geq  \frac{n \ln2}{2}  \,  \ge 3\epsilon n+\frac{ \sqrt{\epsilon} \ln2}{2\beta q}h.
\end{align*}
Thus~\eqref{eq-high-left-sufficient} holds in both cases, and hence
$M_{\mathrm{A}}/\ln2\geq h/(2q)$.

\item Finally,
since $b\geq\Phi(v)\geq\Phi(\beta)$ and $h\geq\delta n$, we can upper bound the left-hand side in~\eqref{eq-level-set-comparison-right-endpoint} using $(2\beta-q)\delta>\beta\delta>6\epsilon$, $\beta<2$, and $\delta\geq4\epsilon^{1/4}$:
\begin{align*}
[\beta^2-q^2]n^2-(2\beta-q)hn+6\epsilon n^2\leq(\beta^2-q^2)n^2 \leq2\beta\sqrt\epsilon\,n^2
 \leq\frac{\delta^2n^2}{4}
 \leq\frac{h^2}{4}.
\end{align*}
\end{itemize}

This proves~\eqref{eq-level-set-comparison-right-endpoint} and completes the proof.
\end{proof}

\begin{proof}[Proof of Lemma \ref{lem-level-set-event-comparison}]
Suppose that $A^{[x]}_{n,L,M_{\mathrm{A}}}$ occurs. Then there exists some $u\in\mathbb{T}_{\leq n}$ such that, with
$
 t\coloneqq n-|u|$ and $
 \Lambda \coloneqq L-(x+S_u)$, 
we have $\varPhi(t,\Lambda)\geq M_{\mathrm{A}}$. We claim that  
\begin{align*}
 \varPhi(t,\Lambda+h)
 =\varPhi(t,\Lambda)-h\frac{\Lambda+h/2}{t} \geq M_{\mathrm{A}} - q h . 
\end{align*}
Indeed, once this claim is proved, \eqref{eq-level-set-comparison-height} yields $\varPhi(t,\Lambda+h) \ge M_{\mathrm{B}} $.
Consequently, the same vertex $u$ triggers the event $A^{[x]}_{n,L+h,M_{\mathrm{B}}}$. This gives the desired relation~\eqref{eq-level-set-comparison-disjoint}.

It remains to prove the claim. It suffices to verify
$
 \Lambda+{h}/{2}\leq qt$.
Since $\varPhi(t,\Lambda)\geq M_{\mathrm{A}}$, necessarily $|u|<n$ and
$
 \frac{M_{\mathrm{A}}}{\ln2}\leq t\leq n$.
Condition~\eqref{eq-level-set-comparison-left-endpoint} therefore gives $qt-{h}/{2} \ge 0$. So it is enough to show 
$\Lambda^2 \le (q t- h/2)^2$. From $\varPhi(t,\Lambda)\geq M_{\mathrm{A}}$ we deduce $\Lambda^2 \le 2 t (t \ln 2- M_{\mathrm{A}}  )$. Hence it suffices to show that
\begin{equation}\label{eq-level-set-comparison-positive}
 2t ( t\ln2-M_{\mathrm{A}} )\leq(qt-h/2)^2  .
\end{equation}

Set $t_0\coloneqq M_{\mathrm{A}}/\ln2$ and consider the function 
\[
 F(s)\coloneqq 2s(s\ln2-M_{\mathrm{A}})-(qs-h/2)^2.
\]
Since $q<\beta_c=\sqrt{2\ln2}$, the function $F$ is convex. Note that $F(t_0)=-(qt_0-h/2)^2\leq0$, while condition~\eqref{eq-level-set-comparison-right-endpoint} is precisely $F(n)\leq0$. Hence $F(s)\leq0$ for every $s\in[t_0,n]$. This completes the proof. 
\end{proof}

\subsection{Proof of Theorem~\ref{thm-Z-tail}: derivation of the super-exponential decay}
\label{sec-pf-thm1}
In this section, we derive the desired super-exponential decay rate by analyzing the growth of high-order moments of $(Z_{\infty}(\beta))_+$. Recall from \eqref{eq-tail-of-D} that the left tail of  $Z_{\infty}(\beta)$ exhibits polynomial decay. We thus restrict ourselves to the positive part $(Z_{\infty}(\beta))_+$. 

\begin{lemma}[{\cite[Proposition 2.7.1]{Ver18}}]\label{lem-sub-exp}
Let $X\ge 0$ be a random variable. Then the following properties are equivalent, with $K_1$ and $K_2$ differing only by universal multiplicative constants:
\begin{enumerate}[label=(\arabic*)]
  \item For some $K_1>0$, the right tail of $X$ satisfies $\P(X \geq t) \leq 2\exp \left(-t/K_1\right)$ for all $t\geq0$.
  \item For some $K_2>0$, the moments of $X$ satisfy $ \E[ X^k ]  \leq K_2^k \,  k^k $ for all $ k \geq 1$.
\end{enumerate} 
\end{lemma}
 
\begin{proof}[Proof of Theorem~\ref{thm-Z-tail}: Upper bound]
We proceed in four steps. 
\smallskip

\noindent
\underline{\textit{Step 1.}}  
To prove the upper bound in Theorem \ref{thm-Z-tail}, we need to show that $ (Z_{\infty}(\beta))_+^{\gamma } $ has sub-exponential tails. By applying Lemma \ref{lem-sub-exp}, it is reduced to proving a uniform bound on the moments. In particular, we show that for all $k\geq 1$, $\sup_{x\in\mathbb{R}}\E[(Z_{\infty}^{[x]}(\beta))_+^{\gamma k}]\leq C_1^k k^{k}$, or equivalently,
\begin{equation}\label{eq-Z-k-moment}
  \sup_{x \in \mathbb{R}}\E[  (Z_{\infty}^{[x]}(\beta))_+^{ k} ]  \leq  C_2^k k^{k/\gamma} 
\end{equation} 
for some positive constants $C_{1},C_{2}$. This estimate and Lemma \ref{lem-sub-exp} yields Theorem~\ref{thm-Z-tail} and Remark~\ref{rmk:enhanced-thm}.

We first note that Lemma~\ref{lem-polynomial-decay} already yields a preliminary moment estimate, namely the finiteness of the k-th moment for every $k\geq 1$:
\begin{equation}\label{eq-Z-k-moment-finite}
  \sup_{x \in \mathbb{R}}\E[  (Z_{\infty}^{[x]}(\beta))_+^{ k} ]  < \infty.
\end{equation}

\noindent
\underline{\textit{Step 2.}}  
We  derive a recursive relation for $ Z^{[x]}_{\infty} (\beta)$. Notice that   
\begin{align}
  Z^{[x]}_{n+1} (\beta)
  & = \frac{e^{\beta^2/2}}{2}
  \sum_{|u|=1}
  \frac{1}{2^n}\sum_{v\in\mathbb{T}^{(u)}_{n}}
  \big[\beta n-\bigl(x+S_u-\beta+S^{(u)}_v\bigr)\big]
  e^{\beta(x+S_u-\beta+S^{(u)}_v)-\frac{\beta^2}{2}n} \\
    & = q \sum_{|u|=1} Z^{[x+S_u-\beta],(u)}_{n}(\beta) 
\end{align}
where $q\coloneqq \frac{e^{\beta^2/2}}{2} =e^{-(1-\frac{1}{\gamma})\ln 2} \in (0,1) $.  Taking the limit as $n \to \infty$  and considering the positive part yields the recursive inequality:
\begin{equation}
  (Z^{[x]}_{\infty} (\beta) )_+ \le  q \sum_{|u|=1} \left( Z^{[x+S_u-\beta],(u)}_{\infty}(\beta) \right)_+.
\end{equation} 

\smallskip
\noindent
\underline{\textit{Step 3.}} 
To prove the moment bound \eqref{eq-Z-k-moment}, we use the above recursive inequality.  Raising the above inequality to the $k$-th power and taking the expectation, we obtain by independence,
\begin{equation}
  \E  \left[ (    Z^{[x]}_{\infty} (\beta))_+^k   \right] \leq q^k\sum_{j=0}^{k} \binom{k}{j} \E \biggl[  \Bigl( Z^{[x+\xi-\beta]}_{\infty}(\beta) \Bigr)_+^j  \biggr]   \E \biggl[   \Bigl( Z^{[x+\xi-\beta]}_{\infty}(\beta) \Bigr)_+^{k-j}    \biggr] .
\end{equation} 
Normalizing by $k!$ turns the binomial sum into a convolution: 
\begin{equation}
  \frac{  1 }{k!}\E  \left[ (    Z^{[x]}_{\infty}(\beta))_+^k   \right] \leq q^k  \sum_{j=0}^{k}  \frac{  1}{j!} \E \left[  \left( Z^{[x+\xi-\beta]}_{\infty}(\beta) \right)_+^j  \right]  \frac{ 1  }{(k-j)!}\E \left[   \left( Z^{[x+\xi-\beta]}_{\infty}(\beta) \right)_+^{k-j}    \right].
\end{equation}
To simplify this expression, we define a rescaled sequence:
\begin{equation}
  t(k) \coloneqq \sup_{x\in\R} \,  \frac{1 }{k !} e^{(1-\frac{1}{\gamma})k \ln k} \E \left[  \left( Z^{[x]}_{\infty}(\beta) \right)_+ ^k  \right] \in (0,\infty)  . 
\end{equation}
where the finiteness of $t(k)$ follows from~\eqref{eq-Z-k-moment-finite}. With $q=e^{-(1-\frac{1}{\gamma})\ln 2}$, substituting this definition into the convolution inequality yields:
\begin{equation}
  t(k) \leq \sum_{j=0}^k e^{(1-\frac{1}{\gamma}) [k \ln \frac{k}{2} - j \ln j - (k-j)\ln(k-j) ] } t(j) t(k-j),
\end{equation}
where we use the convention that $0 \ln 0 =0$. The exponent can be bounded by observing that the function $x \mapsto x \ln x +(k-x) \ln (k-x)$ takes its minimum at $x=k/2$. It follows that for all $1 \leq j \leq k-1$ there holds
$
  k \ln \frac{k}{2} - j \ln j - (k-j)\ln(k-j)  \leq 0.
$
The upshot is the recursive relation for $t(k)$:
\begin{equation}\label{eq-t(k)-recursive-ineq}
  t(k) \leq  2 q^k  t (k) + \sum_{j=1}^{k-1} t(j) t(k-j).
\end{equation}

\smallskip
\noindent
\underline{\textit{Step 4.}} 
We finally derive an upper bound on $t(k)$ by induction. Define  
\begin{equation}
  A\coloneqq  \sup_{k \geq 1} \,    k^2 \, \sum_{j=1}^{k-1} \frac{1}{j^2} \frac{1}{(k-j)^2}  < \infty .
\end{equation}
We now prove that there exists a constant $C_{3}>0$ depending only on $\beta$ such that for all $k\geq 1$,
\begin{equation}\label{eq-t(k)-induction}
  t (k) \leq  \frac{1}{2Ak^2} C_{3}^k.
\end{equation}  

Let $k_0$ be the minimal integer $k\geq 1$ satisfying that $1-2q^k\geq\frac{1}{2}$. Since $t(j)$ is finite for any $j$, we can choose a constant $C_{3}$ large enough such that the inequality~\eqref{eq-t(k)-induction} holds for all $j\le k_0$. This is our base case of the induction. Assume now that~\eqref{eq-t(k)-induction} holds for all integers up to $k-1$ where $k>k_0$.  From the recursive relation \eqref{eq-t(k)-recursive-ineq} and the definition of $A$, it follows that: 
\begin{equation}
  t(k)\leq\frac{1}{(1-2q^k)}\frac{1}{4A^2k^2}k^2\sum_{j=1}^{k-1}\frac{1}{j^2}\frac{1}{(k-j)^2}C_{3}^k\leq\frac{1}{2Ak^2}C_{3}^k.
\end{equation} 
This completes the induction. Substituting the definition of $t(k)$ into~\eqref{eq-t(k)-induction} yields:   
\begin{equation}
  \E \left[  \left( Z_{\infty}^{[x]}(\beta) \right)_+^k\right]\leq\frac{1}{2Ak^2} C_{3}^k \,   k! e^{-(1-\frac{1}{\gamma})k \ln k} \le  \frac{1}{2A} C_{3}^k \, k^{\frac{1}{\gamma} k }  .
\end{equation} 
This is the desired moment condition~\eqref{eq-Z-k-moment} and completes the proof.
\end{proof}

 \subsection{Proof of Theorem~\ref{thm-Z/W-tail}}\label{sec-pf-thm2}
For simplicity, in the following  we write 
$R_{\beta}\coloneqq \frac{Z_{\infty}(\beta)}{W_{\infty}(\beta)} $.

 \begin{proof}[Proof of Theorem~\ref{thm-Z/W-tail}]
  Observe that 
\[
 Z_{\infty}^{[x]}(\beta)
 =
 e^{\beta x}\bigl[ Z _{\infty}(\beta)-xW_{\infty}(\beta)\bigr]
 =
 e^{\beta x}W_{\infty}(\beta) \bigl[  R_{\beta}-x\bigr] \ , \ \forall x \in \mathbb{R} .
\] 
Therefore, given any $x \ge 1$, we have 
\[
 \P(R_{\beta}\ge 2 x)
 \le
 \P\bigl(Z_{\infty}^{[x]}(\beta)\ge x \bigr)
 +
 \P\bigl(W_{\infty}(\beta)<e^{-\beta x}\bigr).
\]

The first probability on the right-hand side is bounded  by $\exp(-\Theta(x^{\gamma}))$ according to \eqref{eq:enhanced-thm}. For the second
term, applying the small deviation probability for $W_{\infty}(\beta)$, we have  
\begin{equation}
  \label{eq-small-ball-W}
   \P\bigl(W_{\infty}(\beta)<e^{-\beta x}\bigr)
   \le \exp(- \Theta  (x^{2 }) ).
\end{equation}
Therefore $\P(R_{\beta}\ge 2 x) \lesssim \exp (- \Theta(x^{\gamma \wedge 2}) )$, which proves  Theorem~\ref{thm-Z/W-tail}.
 
Here we provide a simple proof of \eqref{eq-small-ball-W} using the argument in \cite[Lemma A.3]{BS09}.  Using the  relation $
W_{\infty}(\beta) = \frac{1}{2} \sum_{|u|=1} e^{\beta \xi_{u}- \frac{1}{2}\beta^2} W_{\infty}^{(u)}(\beta)$
and   \(\frac{x+y}{2}\ge \sqrt{xy}\), we get
\[
\E[W_{\infty}(\beta)^{-q}] \le \bigl( \E [ e^{-q \beta \xi/2 +  \beta^2q/2 }  ] \bigr) ^2 \bigl( \E[W_{\infty}(\beta)^{-\frac{q}{2}}] \bigr)^2  = e^{ \frac{\beta^2}{2}q+\frac{\beta^2}{4}q^2 } \bigl( \E[W_{\infty}(\beta)^{-\frac{q}{2}}] \bigr)^2. 
\] 
Starting from  the fact that  \( \E[W_{\infty}(\beta)^{-1} ]<\infty\), iteration yields a constant \(C <\infty\) such that
\[
 \E[W_{\infty}(\beta)^{-q}]   \le \exp\{C q^2\} \ , \ \forall \ q\ge1.
\]
This implies $\ln (1/ W_{\infty})$
 is sub-Gaussian (see e.g. \cite[Proposition 2.5.2]{Ver18})
and~\eqref{eq-small-ball-W} follows.
\end{proof}
  
%%%%%%%%%%%%%%%%%%%%%%%%%%%%%%%%%%
%% APPENDIX %%%%%%%%%%%%%%%%%%%%%
%%%%%%%%%%%%%%%%%%%%%%%%%%%%%%%%%%

\appendix 
 
\section{Proof of Lemmas in Section~\ref{sec-LDP}}\label{Sec-LDP-proof}

\subsection{Proof of Lemma~\ref{lem-LDP-level-set-below-the-line}}\label{sec-LDP-proof-b}

We begin by recalling a useful inequality from \cite{AHS19}. 
Let  $(X_n)_{n\geq 0}$ be an inhomogeneous Galton-Watson process.  
At generation $n$  each individual independently  produces a number of offspring according to a common $\mathbb{N}$-valued random variable $\nu_n$. Its mean is denoted by:
\begin{equation}
  m_n\coloneqq\E\left[ \nu_n\right] \in(0,\infty). 
\end{equation} 
  
\begin{lemma}[{\cite[Proposition 2.1]{AHS19}}]\label{lem-IGW-tail}
Let $\alpha>1$ and $n\geq 1$. If there exists $\lambda>0$ such that for all $i \geq0$,
\begin{equation}\label{eq-IGW-tail-cond}
  \E\left[ e^{\lambda\nu_i}\right] \leq e^{\alpha\lambda m_i},
\end{equation} 
then for all $h >0$ and all integers $\ell\geq1$, 
\begin{equation}
  \P\bigg(X_n\geq\max\bigg\{\ell,(\alpha+h)^n\ell\max_{0\leq i<n}\prod_{j=i}^{n-1}m_j\bigg\}\Big|X_0=\ell\bigg)\leq n\exp\bigg(-\frac{h \ell}{\alpha+h}\lambda+\lambda\bigg).
\end{equation}
\end{lemma} 

We are now ready for the proof of Lemma~\ref{lem-LDP-level-set-below-the-line}.

\begin{proof}[Proof of Lemma~\ref{lem-LDP-level-set-below-the-line}]
  The case $M>n\ln2$ is trivial since $\L^{[x]}_n(L,L+1)\leq2^n$.   Let $L' = L-x$.  Observe that $ \L^{[x]}_{n} (L, L+1 ) =  \L_{n} (L',L'+1 )$ and $B^{[x]}_{n, L,M}  =  B_{n,L',M}$. 
  We can therefore assume, without loss of generality, that $x=0$, $M \le n \ln 2$. 

\smallskip
\noindent
\underline{\textit{Step 1.}}
 Fix $\delta \in (\frac{3}{4},1)$. We discretize time by partitioning $\{0,\cdots, n\}$ into   blocks of length $n^{\delta}$. Set 
 \begin{equation}
   N\coloneqq n^{1-\delta} ~\text{and}~ \mathsf{s}_{j} \coloneqq j n^{\delta} \ \text{for}\  j=0,1,2,\cdots,N.
 \end{equation}   
Next, we introduce a  good  event $E(n)$ where the BRW does not travel too far from the origin:
 \begin{equation}
  E(n) \coloneqq\{ \forall \, u \in \mathbb{T}_{\le n}: |S_{u}| \leq  2 n^{3/2}   \}.
 \end{equation} 
Using a union bound and a standard Gaussian tail estimate  $\P(\xi>x)\le(2\pi)^{-1/2}x^{-1}e^{-x^2/2}$, it follows 
 \begin{equation}\label{eq-E(n)-c}
  \P( E(n)^{c} ) \leq \sum_{k=1}^{n} 2^{k} \P( \sqrt{k} |\xi | \geq    n^{3/2} ) \leq \frac{1}{n}\sum_{k=1}^{n} 2^{k}  e^{- \frac{2 n^3}{ k}}  \leq 2^{n} e^{- 2 n^{2} }, 
 \end{equation}
where we have used $2^{k}e^{-\frac{2 n^3}{k}}\leq 2^{n}e^{-2n^2}$ for all $1\leq k\leq n$. Observe that on the event $E(n)$,  $\L_{n}(L,L+1) \geq 1$ can only occur if $|L|\leq 1+2n^{3/2}$.  

Next, choose $\delta'\in(0,2\delta-\frac{3}{2})$ and partition the spatial interval $[-2n^{3/2},2n^{3/2}]$ into small intervals of length $n^{\delta'}$. Set $\mathsf{f}_{j}\coloneqq j n^{\delta'}$ for $-2n^{3/2-\delta'}\leq j\leq 2n^{3/2-\delta'}$. This grid forms the basis for a set of discrete paths:
\begin{equation}
  \mathsf{Path}\coloneqq \left\{f: \{\mathsf{s}_{j}\}_{j=0}^{N}\to\{\mathsf{f}_{j}\}_{j=-2n^{3/2-\delta'}}^{2n^{3/2-\delta'}}~\text{and}~f(0)=0, f(n)=L\right\}.
\end{equation}  
Given  $f\in\mathsf{Path}$, we say that a particle $u$ at generation $k$ follows the path $f$ until time $s_k$, if for all $0\leq j\leq k$, the ancestor of $u$ at generation $s_j$ lies in the interval $[f(\mathsf{s}_{j})- n^{\delta^{\prime}}, f(\mathsf{s}_{j})+n^{\delta^{\prime}}]$. Define 
\begin{equation}
  \mathsf{Z}_{k}(f) \coloneqq \sum_{|u|=k}\ind{u~\text{follows}~f~\text{until}~\mathsf{s}_{k}} .
\end{equation}  
Consequently, on the event $E(n)$ we have $\L_{n}(L,L+1)\leq\sum_{f\in\mathsf{Path}}\mathsf{Z}_{N}(f)$.  
 
\smallskip
\noindent
\underline{\textit{Step 2.}} 
We improve this bound by restricting the sum to a smaller set of relevant paths. Observe that $\mathsf{Z}_{k}(f) \geq 1$ only when there exists a particle $u$ at some generation $\mathsf{s}_{k}$ whose displacement satisfies $S_{u}-f(\mathsf{s}_{k})\in[-n^{\delta'}, n^{\delta'}]$. Thus on the event $E(n)\cap B_{n,L,M-\epsilon n}$, any $f\in\mathsf{Path}$ with $Z_{N}(f)\geq 1$ satisfies
\begin{equation}
 \forall 0\leq k<N, \exists h\in [- n^{\delta'}, n^{\delta'}],\quad (n-\mathsf{s}_{k})\ln 2-\frac{[L-(f(\mathsf{s}_{k})+h)]^2 }{2(n-\mathsf{s}_{k})}\leq M-\epsilon n.
\end{equation}
 Since $|L- f(s_{k})|\leq 4 n^{3/2}$ and $n-s_{k}\geq n^{\delta}$, this implies that for all $0\leq k<N$,
\begin{equation}\label{cond-below-the-curve}
    (n-\mathsf{s}_{k}) \ln 2 -  \frac{[L-   f(\mathsf{s}_{k})   ]^2 }{2(n- \mathsf{s}_{k})}  
    \leq M  - \epsilon n + O(n^{3/2+\delta'-\delta}).
\end{equation}  
Hence, on the event $E(n)\cap B_{n,L,M-\epsilon n}$,  we have 
\begin{equation}\label{eq-decom-path}
  \L_{n}(L,L+1)\leq\sum_{f\in\mathsf{Path}^*}\mathsf{Z}_{N}(f),
\end{equation}
where $\mathsf{Path}^*\coloneqq\{f\in\mathsf{Path}: f~\text{satisfies}~\eqref{cond-below-the-curve}\}$. By union bound, we estimate the desired probability: 
\begin{equation}
  \P(\{\L_{n}(L,L+1)\geq e^{M}\}\cap B_{n,L,M-\epsilon n})\leq\P(E(n)^{c})+\sum_{f\in\mathsf{Path}^{*}}\P\left(\mathsf{Z}_{N}(f)\ge \frac{ e^{M} } {  |\mathsf{Path}^{*}|} \right).
\end{equation} 
Fix $\epsilon'<\epsilon/2$. Observe that $|\mathsf{Path}^{*}|\leq|\mathsf{Path}|\leq(2n^{3/2})^{N}=e^{o(n)}=o(e^{\epsilon'n})$. We will prove that for sufficiently large $n$ depending only on $\delta,\epsilon$ and $\epsilon'$,  
\begin{equation}\label{eq-Pr-f}
  \mathrm{Pr}_{\eqref{eq-Pr-f}}(f)\coloneqq\P(\mathsf{Z}_{N}(f)\geq e^{M-\epsilon'n})\leq\exp\left(-e^{n\epsilon'/2}\right) \ , \ \forall\, f\in\mathsf{Path}^* . 
\end{equation}  
Together with the previous inequalities, we get the desired result
\begin{align*}
  \P(\{\L_{n}(L,L+1)\geq e^{M}\}\cap B_{n,L,M-\epsilon n}) \leq 2^{n}e^{-2n^{2}}+(2n^{3/2})^{N}\exp\left(-e^{n\epsilon'/2}\right) =o(e^{-n^{2}}).
\end{align*}  
 
\smallskip
\noindent
\underline{\textit{Step 3.}} 
It remains to show \eqref{eq-Pr-f}. 
On the event $\mathsf{Z}_{N}(f)\geq e^{M-\epsilon'n}$, there exists a smallest index $r$ such that $\mathsf{Z}_{r}(f)\geq e^{\epsilon'n}$. From the fact that $\mathsf{Z}_{r-1}(f)<e^{\epsilon'n}$ and the binary branching mechanism, it follows that $\mathsf{Z}_{r}(f)<2^{n^{\delta}}e^{\epsilon'n}$. Particularly this implies $r<N$. 
Employing the union bound, we get 
\begin{equation}\label{eq-Pr-f-1}
  \mathrm{Pr}_{\eqref{eq-Pr-f}}(f)\leq\sum_{r=1}^{N-1} \ \sum_{\ell=e^{\epsilon'n}}^{2^{n^{\delta}}e^{\epsilon'n}}\P(\mathsf{Z}_{N}(f)\geq e^{M-\epsilon'n},\mathsf{Z}_{r}(f)=\ell).
\end{equation}

Conditionally on $\{\mathsf{Z}_{r}(f)=\ell\}$, for each $r\leq j\leq N-1$, we can express $\mathsf{Z}_{j+1}(f)$ as   $\sum_{i=1}^{\mathsf{Z}_j(f)} \nu_i^{(j)}$   where $\nu_i^{(j)}$ denotes the number of offspring of the $i$-th individual counted in $\mathsf{Z}j(f)$  located within  $[f(\mathsf{s}_{j+1}) - n^{\delta'}, f(\mathsf{s}_{j+1}) + n^{\delta'}]$.  
By the branching property, $(\nu_i^{(j)})_{i\geq 1}$ are independent. They would further be identically distributed if the particles at generation $\mathsf{s}_{j}$ were exactly positioned at $f(\mathsf{s}_{j})$. 
However, $\nu_i^{(j)}$ is stochastically dominated by  $\widetilde{\nu}^{(j)}$, 
defined as the number of particles at time $n^\delta$ in a BRW starting at the origin, which fall into  $[ \Delta f(\mathsf{s}_{j+1})-2n^{\delta^{\prime}},\Delta f(\mathsf{s}_{j+1})+2n^{\delta^{\prime}} ]$ at time $n^\delta$ where $\Delta f(\mathsf{s}_{j+1})\coloneqq f(\mathsf{s}_{j+1})-f(\mathsf{s}_{j})$. 
Consequently, conditioned on  $\{\mathsf{Z}_{r}(f)=\ell \} $, $\mathsf{Z}_{N}(f)$ is dominated by the population size $X_{N-r}$ of an inhomogeneous Galton-Watson process at generation $N-r$ with offspring distribution sequence $(\widetilde{\nu}^{(r+j)})_{j\ge 0}$ and initial population $X_0=\ell$. This domination directly implies: 
\begin{equation}\label{eq-Pr-f-2}
  \P(\mathsf{Z}_{N}(f)\geq e^{M-\epsilon'n} \mid \mathsf{Z}_{r}(f)=\ell)\leq\P\left(X_{N-r}\geq e^{M-\epsilon'n}\mid X_0=\ell\right).
\end{equation}

In order to apply Lemma~\ref{lem-IGW-tail} to the right-hand side in \eqref{eq-Pr-f-2}, we now select a convenient $\lambda$ satisfying condition~\eqref{eq-IGW-tail-cond}.  Set $\lambda= 2^{-n^{\delta}}$ so that $\lambda\widetilde{\nu}^{(j)}\le1$, and let $\alpha=e-1$.   Since $e^{x}\leq 1+\alpha x$ for all $x\in[0,1]$,  it follows that $\E(e^{\lambda\widetilde{\nu}^{(j)}})\leq 1+\alpha\lambda\E(\widetilde{\nu}^{(j)})\leq e^{\alpha\lambda\E(\widetilde{\nu}^{(j)})}$. Next, we need to bound the product term $\max_{r\leq k<N}\prod_{j=k}^{N-1}\E(\widetilde{\nu}^{(j)})$.  First, we estimate the mean of a single offspring distribution: 
\begin{align}
  \E(\widetilde{\nu}^{(j)})&=2^{n^\delta}\P\left(\left| n^{\delta/2} \xi -\Delta f(\mathsf{s}_{j+1})\right|\leq 2n^{\delta^{\prime}}\right)\\
  &\leq\exp\bigg(n^\delta\ln 2-\frac{\left(\Delta f(\mathsf{s}_{j+1})\right)^2}{2n^\delta}+2|\Delta f(\mathsf{s}_{j+1})|n^{\delta'-\delta}+O\left(\ln n\right) \bigg)
\end{align}
where $O\left(\ln n\right)$ is uniform in $r,j$, and $f$. Since $f\in\mathsf{Path}^*$, we know $|\Delta f(\mathsf{s}_{j+1})|\leq 4n^{\frac{3}{2}}$. Thus for any $r \le k <N$, the product is bounded by:
\begin{equation} 
  \prod_{j=k}^{N-1}\E(\widetilde{\nu}^{(j)})\leq\exp\bigg((n-\mathsf{s}_{k})\ln 2-\frac{1}{2 n^\delta}\sum_{j=k}^{N-1}\left(\Delta f(\mathsf{s}_{j+1})\right)^2+O(n^{\frac{5}{2}+\delta'-2\delta})\bigg).
\end{equation} 
By using the Cauchy-Schwarz inequality, the sum in the exponent has a lower bound:
\begin{equation}
  \frac{1}{2n^\delta}\sum_{j=k}^{N-1}\left(\Delta f(\mathsf{s}_{j+1})\right)^2\geq\frac{\left(f(\mathsf{s}_{N})-f(\mathsf{s}_{k})\right)^2}{2n^{\delta}(N-k)}=\frac{\left[ L-f(\mathsf{s}_{k})\right]^2}{2(n-\mathsf{s}_{k})}.  
\end{equation}  
Combined with \eqref{cond-below-the-curve}, this gives
\begin{equation}
  \prod_{j=k}^{N-1}\E(\widetilde{\nu}^{(j)})\leq\exp\left(M-\epsilon n+O(n^{\frac{5}{2}+\delta'-2\delta})+O(n^{\frac{3}{2}+\delta'-\delta})\right).
\end{equation} 
Our choice of  $\delta'$  ensures that the exponents in the error terms are strictly less than 1, specifically 
$0<\frac{3}{2}+\delta'-\delta<\frac{5}{2}+\delta'-2\delta<1$.  
Recalling that 
$M\geq\epsilon n>2\epsilon'n$ and $e^{\epsilon'n}\leq\ell\leq 2^{n^{\delta}}e^{\epsilon'n}$,  we deduce  
\begin{equation}
  \max\bigg\{\ell,(\alpha+1)^{N-r}\ell\max_{r\leq k<N}\prod_{j=k}^{N-1}\E(\widetilde{\nu}^{(j)})\bigg\}\leq\max\{e^{\epsilon'n+o(n)},e^{M-(\epsilon-\epsilon')n+o(n)}\}=o(e^{M-\epsilon'n}). 
\end{equation} 
Applying Lemma~\ref{lem-IGW-tail}  
with parameters $\alpha=e-1$, $h=1$ and $\lambda=2^{-n^{\delta}}$ we obtain: 
\begin{equation}\label{eq-Pr-f-3}
  \P\left(X_{N-r}\geq e^{M-\epsilon'n}\mid X_0=\ell\right)\leq n\exp\left(-  2^{-n^{\delta}} {\ell}/{e} + 2 ^{-n^\delta}\right)  \  \forall \, 1\leq r<N .
\end{equation}
Finally, combining inequalities \eqref{eq-Pr-f-1}, \eqref{eq-Pr-f-2}, and \eqref{eq-Pr-f-3},   
for large $n$ depending only on $\delta,\epsilon$ and $\epsilon'$,
\begin{equation}
  \mathrm{Pr}_{\eqref{eq-Pr-f}}(f)\le n\sum_{r=1}^{N} \ \sum_{\ell=e^{\epsilon'n}}^{2^{n^{\delta}}e^{\epsilon'n}}\exp\left(-2^{-n^{\delta}} {\ell}/{e}+  2 ^{-n^\delta}\right)\leq\exp\left(-e^{ n\epsilon'/2}\right) 
\end{equation}
as required in \eqref{eq-Pr-f}. 
\end{proof}

\subsection{Proof of Lemma~\ref{lem-LDP-level-set-above-the-line}}\label{sec-LDP-proof-a}

\begin{proof}[Proof of Lemma~\ref{lem-LDP-level-set-above-the-line}]
For $\epsilon>0$ and $0\leq M\leq n(\ln 2-\epsilon)$,  let us define the set
\begin{equation}
  \mathcal{P}_{n,\epsilon,M}\coloneqq\{(t,\Lambda): 1\leq t\leq n, \varPhi(t,\Lambda)\geq M+\epsilon n\}. 
\end{equation} 
Notice that for any $(t,\Lambda)\in\mathcal{P}_{n,\epsilon, M}$, we have $|\Lambda|\leq t\sqrt{2(\ln 2-\epsilon)}$ and $t\geq \frac{\epsilon}{\ln 2}n\geq\epsilon n$. 

First, we show that it is sufficient to establish the following uniform bound:
\begin{equation}\label{eq-restart-t-L}
  \mathrm{Pr}_{\eqref{eq-restart-t-L}}(n,\epsilon)\coloneqq\sup_{0\leq M\leq n(\ln 2-\epsilon)}\sup_{(t,\Lambda)\in\mathcal{P}_{n,\epsilon,M}}\P\left(\L_{t}(\Lambda,\infty)\le e^M\right)\le e^{-cn ^2}.
\end{equation}
 To see why this is sufficient, let $\tau$ denote   the minimal integer $k$ such that there exists a particle $u$ at generation $k$ for which the event $A^{[x]}_{n,L,M+\epsilon n}$ occurs.
  Let $u^*$ be the lexicographically smallest such particle. In particular $(|u^*|, [x+S_{u^*}] ) \in  \mathcal{P}_{n,\epsilon,M}$.   On the event $\{\L^{[x]}_{n}(L,\infty)\leq e^{M}\}$, we have $\L^{(u^*)}_{n-\tau}(L-[x+S_{u^*}],\infty)\le e^M$.
It follows from the strong Markov property that 
\begin{equation}
  \P\left(\L^{(u^*)}_{n-\tau}(L-[x+S_{u^*}],\infty)\le e^M~|~\mathcal{F}_{\tau}\right)\ind{A^{[x]}_{n,L,M}}\le\mathrm{Pr}_{\eqref{eq-restart-t-L}}(n,\epsilon).
\end{equation}
Taking the expectation yields the claim.

We now turn to proving \eqref{eq-restart-t-L}. 
Take $t_n=\frac{3}{\ln 2}\ln n$ such that $2^{t_n}=n^3$, and choose $\delta=\frac{\epsilon^2}{4\beta_{c}}\wedge\beta_{c}$. For $n$ sufficiently large (depending only on $\epsilon$), we have the estimate:
\begin{align}\label{eq-modulus-conti-varPhi}
  |\Phi(t-t_n,\Lambda+\delta n)-\Phi(t,\Lambda)|&\leq t_n\ln 2+\frac{|\Lambda|\delta n+\delta^2 n^2}{2t}+\frac{|\Lambda+\delta n|^2 t_n}{2t(t-t_n)}\\
  &\leq\Theta_{\epsilon}(\ln n)+\frac{\beta_c+\delta}{2\epsilon}\delta n \leq\Theta_{\epsilon}(\ln n)+\frac{\epsilon}{4}n.
\end{align}
This implies that  $(t-t_n,\Lambda+\delta n)\in\mathcal{P}_{n,\epsilon/2,M}$. Applying the branching property at generation $t_n$,
\begin{align}
  \P(\L_{t}(\Lambda,\infty)\le e^M~|~\mathcal{F}_{t_n})   
  &\leq\prod_{|u|=t_{n},S_u\geq -\delta n}\P\left(\L_{t-t_{n}}^{(u)}(\Lambda-S_u,\infty)\leq e^M~|~\mathcal{F}_{t_{n}}\right)\\
  &\leq\P\left(\L_{t-t_{n}}(\Lambda+\delta n,\infty)\leq e^{M}\right)^{\L_{t_{n}}(-\delta n,\infty)}.
\end{align}
We claim that for large $n$ depending only on $\epsilon$, 
\begin{equation}\label{eq-typical-size}
  \sup_{0\leq M\leq n(\ln 2-\epsilon)}\sup_{(t',\Lambda')\in\mathcal{P}_{n,\epsilon/2, M}}\P\left(\L_{t'}(\Lambda',\infty)\leq e^M\right)\leq 1/2. 
\end{equation} 
Assuming this claim holds, taking the expectation gives 
\begin{equation}
  \P(\L_{t}(\Lambda,\infty)\le e^M)\leq(1/2)^{n^2}+\P\left(\L_{t_{n}}(-\delta n,\infty)\leq n^2\right).
\end{equation} 
It follows from~\cite[(3.36)]{CH20} that there exists a constant $c>0$ such that
\begin{equation}
  \P\left(\L_{t_{n}}(-\delta n,\infty)\leq n^2\right)=\P\left(\L_{t_{n}}(-\delta n,\infty)\leq 2^{\frac{2}{3}t_{n}}\right)\le e^{-cn^2} .   \footnote{The trivial union bound gives   $   \P \left( \L_{t_{n}}(-\infty, -\delta n ) \geq 1 \right) \leq 2^{t_{n}} e^{- \Theta (n^2/ \ln n )} = e^{- \Theta (n^2/ \ln n )}  $ and indeed this is sufficient for our use to deduce the main theorems in this paper.} 
\end{equation}

It remains to show the claim \eqref{eq-typical-size}.  Let $\mathcal{Q}_{\epsilon}\coloneqq\epsilon^3\mathbb{Z}\cap(-\sqrt{2\ln 2},\sqrt{2\ln 2})$. 
For each $q\in\mathcal{Q}_{\epsilon}$, it is well-known that the following Law of Large Numbers holds (see e.g.~\cite{Biggins79}):
\begin{equation}
 \frac{1}{n}\ln\L_{n}(q n,\infty)\xrightarrow[n\to\infty]{\mathrm{a.s.}}\ln 2-\frac{q^2}{2}\ind{q\geq 0}\geq\Phi(q), 
\end{equation} 
Hence there exists a large constant $K(q,\epsilon)$ such that for all $n \geq K(q,\epsilon)$,
\begin{equation}\label{eq-SLLN-2}
  \P(\L_{n}(q n,\infty)\leq e^{n\Phi(q)-n\epsilon/4})\leq 1/2.
\end{equation}
Now, consider any  $M\in[0,n(\ln 2-\epsilon)]$ and $(t',\Lambda')\in\mathcal{P}_{n,\epsilon/2,M}$.  Recall that this implies   $|\Lambda'/t'|\leq \sqrt{2(\ln2-\epsilon)}$ and $t' \ge \epsilon n /2$. Define $q'=\max\{q\in\mathcal{Q}_{\epsilon}: q\leq\Lambda'/t'\}$, in particular $0\leq\Lambda'-qt'\leq\epsilon^3 n$.  We will show that $M\leq t'\Phi(q')-t'\epsilon/4$.  If this holds, then
\begin{equation}
  \{\L_{t'}(\Lambda',\infty)\leq e^{M}\}\subset\left\{\L_{t'}(q't',\infty)\leq e^{t'\Phi(q')-t'\epsilon/4}\right\}.
\end{equation} 
Then, for $n$ large enough so that $t'\geq n\epsilon/2 \geq\max_{q\in\mathcal{Q}_{\epsilon}}K(q,\epsilon)$, it follows from \eqref{eq-SLLN-2} that  the probability of the right-hand side is at most $1/2$. This would prove \eqref{eq-typical-size}. 

To verify the inequality for $M$, recall that  $M \leq \varPhi(t',\Lambda')-n\epsilon/2 $. Moreover we have
\begin{equation}
 | \varPhi(t',\Lambda')- \Phi(q')t'|=   \left|\frac{ \Lambda'^2}{2 t'} - \frac{(q' t')^2 }{2 t'} \right| \leq  \frac{|\Lambda|+|q'|t'}{2 t'} \epsilon^3 n \leq   \beta_{c} \epsilon^3 n \leq \Theta( \epsilon^2 t') .
\end{equation}
For a sufficiently small choice of $\epsilon$, we have $ \beta_{c} \epsilon < \frac{1}{4}$ which gives $M \leq t'\Phi(q') + \beta_{c} \epsilon^2 t' - n\epsilon/2 \leq t'\Phi(q') - t' \epsilon/4$ as required. This completes the proof.
\end{proof}

\section*{Acknowledgments}
X.C. and H.M. express their gratitude to Prof. Jian Song for his generous hospitality and support at the Research Center for Mathematics and Interdisciplinary Sciences, Shandong University, during their research visit.
X.C. is supported by National Key R\&D Program of China (No. 2022YFA1006500).
Y.H. is partially supported by National Key R\&D Program of China (No. 2023YFA1010103) and by NSFC-12301164.

%%%%%%%%%%%%%%%%%%%%%%%%%%%%%%%%%%
%% REFERENCE %%%%%%%%%%%%%%%%%%%%%
%%%%%%%%%%%%%%%%%%%%%%%%%%%%%%%%%%
 
 \bibliographystyle{alpha}
 \bibliography{biblio}

 \end{document}